\documentclass[12pt, reqno]{article}

\usepackage{mathptmx}       
\usepackage{helvet}         
\usepackage{courier}        

\usepackage{makeidx}         
\usepackage{graphicx}        
\usepackage[bottom]{footmisc}

\usepackage{latexsym}
\usepackage{amsmath}
\usepackage{amsfonts}
\usepackage{amssymb}
\usepackage{bm}

\usepackage{url}
\usepackage{algorithm}
\usepackage{algorithmic}
\usepackage[misc,geometry]{ifsym}

\newcommand{\bsa}{{\boldsymbol{a}}}

\newcommand{\bsy}{{\boldsymbol{y}}}

\newcommand{\rmd}{{\mathrm{d}}}

\newcommand{\rmr}{{\mathrm{r}}}

\newcommand{\rmK}{{\mathrm{K}}}

\newcommand{\bbM}{{\mathbb{M}}}

\newcommand{\N}{{\mathbb{N}}} 
\newcommand{\R}{{\mathbb{R}}} 
\newcommand{\RR}{{\mathbb{R}}} 

\DeclareSymbolFont{bbold}{U}{bbold}{m}{n}
\DeclareSymbolFontAlphabet{\mathbbold}{bbold}

\newcommand{\calN}{{\mathcal{N}}}

\usepackage{amsmath}
\usepackage{amsfonts}
\usepackage{amssymb}
\usepackage{amsthm}

\newtheorem{theorem}{Theorem}
\newtheorem{lemma}[theorem]{Lemma}
\newtheorem{corollary}[theorem]{Corollary}
\newtheorem{proposition}[theorem]{Proposition}

\newtheorem{assumption}[theorem]{Assumption}

\begin{document}

\title{Direct and Inverse Results on Bounded Domains for Meshless Methods via Localized Bases on Manifolds
\thanks{Thomas Hangelbroek is supported by grant DMS-1413726 from the National Science Foundation.
Francis Narcowich and Joseph Ward are supported by grant DMS-1514789 from the National Science Foundation.
Christian Rieger is supported by (SFB) 1060 of the Deutsche Forschungsgemeinschaft.}
}

\author{ T. Hangelbroek\and F. J. Narcowich \and  C. Rieger \and J. D. Ward}


\maketitle

\index{Hangelbroek, T.}
\index{Narcowich, F.J.}
\index{Rieger, C.}
\index{Ward, J.D.}

\center{\bf{Dedicated to Ian H. Sloan on the occasion of his 80th Birthday}}

\abstract{
This article develops direct and inverse estimates for certain finite dimensional spaces 
arising in kernel approximation.  
Both the direct and inverse estimates are based on approximation spaces spanned by local Lagrange functions 
which are spatially highly localized.  
The construction of such functions is computationally efficient 
and generalizes the construction given in \cite{HNRW2} for restricted surface splines on $\R^d$.  
The kernels for which the theory applies includes the 
Sobolev-Mat\'ern kernels for closed, compact, connected, $C^\infty$ Riemannian manifolds.
}

\section{Introduction}\label{sec:1}

This article investigates both direct estimates and inverse inequalities for certain finite dimensional spaces 
of functions.  
These spaces are spanned by either Lagrange or local Lagrange functions generated by 
certain positive definite or conditionally 
positive definite kernels.

While the topics of direct and inverse theorems for kernel-based approximation spaces 
have been considered 
in the boundary-free setting by a number of authors
(see  \cite{NWW_Bernstein}, \cite{rieger2008sampling}, 
\cite{MNPW}, \cite{Ward_J}, \cite{griebel2013multiscale} 
as a partial list), 
the results for such theorems on compact domains is less well developed.  
The main results in this article pertain to inverse estimates (Section~\ref{main_results}) 
and direct theorems 
(Section~\ref{quasi_approx}) 
for certain kernel based approximation spaces on compact domains in a fairly 
general setting.

The primary focus of this article pertains to certain positive definite kernels defined 
on a closed, compact, 
connected, $C^\infty$ Riemannian manifold, which will be denoted by $\bbM$ throughout the sequel.  
We restrict to this setting; inverse theorems in the  Euclidean space setting were recently given 
in  \cite{HNRW2}.

Rather than dealing with the standard finite dimensional kernel spaces 
$S(X)=\text{span}_{\xi\in X} k(\cdot,\xi)$, 
where $k(\cdot,\cdot)$ is a positive definite kernel and $X$, the set of centers, 
is a suitably chosen finite set of points, we will consider subspaces of $S(X)$ that are generated 
by Lagrange functions $\{\chi_\xi:\xi\in X\}$, which, 
for certain kernels, are highly localized. 
These subspaces are designed to deal with problems involving a compact domain  
$\Omega \subsetneq \bbM$, 
where $\Omega$ is subject to some mild restrictions discussed in Section~\ref{bounded_domain}.

Specifically, we look at spaces of the form $V_\varXi =\text{span}_{\xi\in\varXi} \chi_\xi$, 
where $\chi_\xi$ is a Lagrange function for $X$, which is assumed to be suitably dense 
in a neighborhood of $\Omega$, and $\varXi$ is a subset of $X$.  
An important feature, perhaps unusual for RBF and kernel approximation, 
is that the centers $X$ used to \emph{construct} the Lagrange functions $\{\chi_\xi:\xi\in X\}$ 
and centers $\varXi$ \emph{defining} the function spaces $V_\varXi$, do not always coincide, 
because $V_\varXi$ comprises only the Lagrange functions associated with $\xi\in \varXi$. 
The spaces $V_\varXi$ differ slightly from $S(X)$ and are important for obtaining inverse estimates 
over $\Omega$. 
We will discuss these spaces in Section~\ref{Cond} and provide inverse estimates 
in Theorem~\ref{full_inverse}.
We also consider, in Section~\ref{local_lagrange}, locally (and efficiently) constructed functions $b_\xi$, 
which we call \emph{local} Lagrange functions. 
These have properties similar to the $\chi_\xi$'s and also to those used in \cite{DR}.  
In Theorem~\ref{lowercomparison_omega}, we give inverse estimates 
for $\tilde{V}_\varXi =\text{span}_{\xi\in\varXi}b_\xi$.
%

\subsection{Overview and Outline}
In Section~\ref{background}, a basic explanation and background on the manifolds and kernels used in 
this article will be given.

The direct and inverse theorems in this paper are associated with two approximation spaces 
$V_\varXi$ and $\tilde{V}_\varXi$.  
In Section~\ref{lag_bernstein}, we introduce the Lagrange basis 
(the functions which form a basis for the space $V_\varXi$) associated with the 
kernels described in Section~\ref{SSS:Kernels}.  
Such Lagrange functions are known to have stationary exponential decay 
(this notion is introduced in Section \ref{SS:Lagrange}). 
To illustrate the power of these highly localized bases, 
we finish the section by providing estimates that control 
the Sobolev norm (i.e. $W^{\sigma}_{2}(\bbM)$) of a function in $V_\varXi$ 
by the $l_2$ norm on the Lagrange coefficients.  
That is, for $s=\sum_{\xi\in\varXi}a_\xi \chi_\xi$ we show
$$\| s\|_{W^{\sigma}_{2}(\bbM)}\leq C h^{d/2 -\sigma}\|(a_\xi )_{\xi\in\varXi}\|_{l^2 (\varXi)}.$$
This estimate is a crucial first step for the inverse estimates.

Section~\ref{local_lagrange} introduces the other stable basis considered in this paper: 
the local Lagrange basis, which generates the space $\tilde{V}_\varXi$.  
Unlike the Lagrange functions, the local Lagrange bases will be shown to be computationally efficient 
while enjoying many of the same (or similar) properties of the Lagrange bases.  
The local Lagrange bases, which generate the spaces $\tilde{V}_\varXi$, 
provide the focal point of this paper.  
We first give sufficient conditions to prove existence and stability
of such a basis, given Lagrange functions with stationary exponential decay.  
The section culminates with Theorem~\ref{main_local_bernstein} which states that 
there is a constant $C$ so that 
for any $s=\sum_{\xi\in\varXi}a_\xi b_\xi$ 
$$\| s\|_{W^{\sigma}_{2}(\bbM)}\leq C h^{d/2-\sigma}\|(a_\xi )_{\sigma\in\varXi}\|_{l^2 (\varXi )}$$
holds.

Section~\ref{Cond} provides lower stability estimates (i.e. bounding $\|s\|_{L_{2}}$ below 
in terms of the coefficients $\|(a_\xi )_{\xi\in\varXi}\|_{l^2}$) for elements of either 
$V_\varXi$ or $\tilde{V}_\varXi$.
Section~\ref{main_results} presents the complete Sobolev inverse estimates for both the spaces 
$V_\varXi$ and $\tilde{V}_\varXi$ in Theorems~\ref{full_inverse} 
and \ref{lowercomparison_omega} respectively.

Finally in Section~\ref{quasi_approx} the direct theorems are given. 
More specifically, both spaces $V_\varXi$ and $\tilde{V}_\varXi$ 
are shown to provide approximation orders 
for functions of varying smoothness. 
For a continuous function $f$ with no known additional orders of smoothness, 
Theorem~\ref{local} shows that both the interpolant
$I_\varXi f$ or the quasi-interpolant $Q_\varXi f$ approximate $f$ pointwise 
at a rate comparable to the pointwise modulus of continuity $\omega(f, Kh|\ln h|,x_0)$ where
$$\omega(f,t,x_0):=\max_{|x-x_0 |\leq t}|f(x)-f(x_0 )|.$$
These are the first pointwise estimates of their kind for RBF approximation schemes.

The next result applies to smoother functions $f$.  
For a point set $\varXi_e$ which is quasi-uniform over the manifold $\bbM$ 
and given kernel $\kappa_m$, we show that the smoothness of $f$ is captured in the estimate
 \[
 \mathrm{dist}_{p,\mathbb{M}}(f, S(\Xi))
 \le 
 Ch^{\sigma} \|f\|_{B_{p,\infty}^{\sigma}}, \ 1\le p\le \infty, \ 0<\sigma \le 2m,
 \]
where  the Besov space $B_{p,\infty}^{\sigma}(\mathbb{M})$ is defined in \eqref{besov_def}.

Our final result shows that optimal $L_\infty$ approximation rates, 
when approximating a smooth function $f$ on $\Omega$ can be obtained from data sites contained 
in a set ``slightly larger'' than $\Omega$.  
The result illustrates the local nature of the bases $\{\chi_\xi \}$ or $\{b_\xi \}$.

Let $f\in C^k (\Omega)$ and let $f_e \in C^k (\bbM)$ be a smooth extension of $f$ to $\bbM$, 
i.e., $f_e|_\Omega=f|_\Omega$.  
Let $\mathcal{S}=\{x\in\bbM\backslash\Omega,\text{ dist}(x,\Omega)\leq Kh\log h^{-1}\}$ 
and $\varXi$ a discrete quasi-uniform set contained in $\Omega\cup\mathcal{S}$ with fill distance $h$.  
Finally let $\varXi_e$ be a quasi-uniform extension of $\varXi$ to all of $\bbM$ as given 
in Lemma ~\ref{Extension}.  
Also let $\kappa_m$ be a kernel as described in Section~\ref{SSS:Kernels} with associated spaces
$$
\tilde{V}_{\varXi_e}
=\text{span}_{\xi\in\varXi_e}\{b_\xi \}\text{ and }\tilde{V}_{\varXi} 
= 
\text{span}_{\xi\in\varXi}\{b_\xi \}.
$$
The result then states that
$
\text{dist}_{\infty,\Omega}(f,\tilde{V}_\varXi )
\sim
\text{dist}_{\infty,\bbM}(f_e ,\tilde{V}_{\varXi_e})
$
-- that is they are within constant multiples of each other. The upshot is that there are several results on estimating 
$\text{dist}_{\infty,\bbM}(f_e ,\tilde{V}_{\varXi_e}).$

\section{Background: Manifolds and kernels}\label{background}

\subsection{The Manifold $\bbM$}\label{SS:manifolds}
As  noted above, throughout this article $\bbM$ is assumed to be a closed,  compact, connected, 
$C^\infty$ Riemannian manifold. 
The metric for $\bbM$, in local coordinates $(x^1,\cdots, x^d)$, will be denoted by $g_{j,k}$ 
and the 
volume element by $\rmd\mu=\sqrt{\det(g_{j,k})}\rmd x^1\cdots \rmd x^d$. 
Such manifolds have the following properties:

\begin{enumerate}
\item \label{completeness} \emph{Geodesic completeness}. 
$\bbM$ is geodesically complete, by the Hopf-Rinow Theorem \cite[Section 7.2]{doC1}. 
Thus, $\bbM$ is a metric space with the distance $\mathrm{dist}(x,y)$ between $x,y\in \bbM$ 
given by the length of the shortest geodesic joining $x$ and $y$. 
The \emph{diameter} of $\bbM$, which is finite by virtue of the compactness of $\bbM$, 
will be denoted by $\rmd_\bbM$. 
The \emph{injectivity radius} $\rmr_\bbM$, \cite[p.~271]{doC1}, 
which is the infimum of the radius of the smallest ball on which geodesic normal coordinates 
are non singular,  
is positive and finite. 
Of course, $\rmr_\bbM\le \rmd_\bbM$. 

\item \label{vol_Lp} \emph{$L_p$ embeddings}. 
For $\Omega\subset \bbM$, we define $\mathrm{vol}(\Omega)=\int_{\Omega}\rmd\mu$. 
In addition, with respect to $d\mu$, the inner product $\langle \cdot,\cdot\rangle$ 
and all $L_p$ norms are defined in the usual way, and these standard embeddings hold:
\[
L_p(\bbM) \subset L_q(\bbM)\ \text{for}\ 1\le q\le p\le \infty
\]

\item \label{bounded_geometry} \emph{Bounded geometry}. 
$\bbM$ has bounded geometry \cite{cheeger_etal1982, triebel1992}, 
which means that $\bbM$ has  a positive injectivity radius and that derivatives of the Riemannian metric 
are bounded (see \cite[Section 2]{HNW} for details). 
This fact already implies the Sobolev embedding theorem, as well as a smooth family 
of local diffeomorphisms (uniform metric isomorphisms), \cite[(2.6)]{HNW}, 
which induce a  family of metric isomorphisms \cite[Lemma 3.2]{HNW} 
between Sobolev spaces on $\bbM$ and on $\RR^d$.

\item \label{item_manifold} \emph{Volume comparisons}. 
Denote the (geodesic) ball centered at $x\in \bbM$ and having radius $r$ by $B(x,r)$, 
where $0<r\le \rmd_\bbM$. 
There exist constants $0<\alpha_\bbM<\beta_\bbM<\infty$  so that, for all $0<r\le \rmd_\bbM$,
\begin{equation}\label{balls}
\alpha_\bbM r^d \le \mathrm{vol}(B(x,r))\le \beta_\bbM r^d.
\end{equation}
This inequality requires the volume comparison theorem of Bishop and Gromov 
\cite{Eschenburg-1987,Grove-1987}. 
See Section~\ref{appendix_vol_comp} for a proof and explicit estimates on $\alpha_\bbM$ and 
$\beta_\bbM$.
\end{enumerate}

\subsubsection{Point sets}
 Given a set $D\subset \bbM$ and a finite set $X \subset D$, 
 we define its  \emph{fill distance} (or \emph{mesh norm}) $h$ 
and the \emph{separation radius} $q$ to be:
\begin{equation} \label{minimal-separation}
 h(X,D):=\sup_{x\in D} {\mathrm{dist}}(x,X)\qquad \text{and}\qquad  q(X):=\frac12 
\inf_{\xi,\zeta\in X, \xi\ne \zeta} 
{\mathrm{dist}}(\xi,\zeta).
\end{equation}  
The \emph{mesh ratio} $\rho:=h(X,D)/q(X)$ measures the uniformity of the 
distribution of $X$ in $D$. If $\rho$ is bounded, then we say that the point set $X$ is 
quasi-uniformly distributed (in $D$), or simply that $X$ is quasi-uniform.

We remark that for quasi-uniform $X$ and any $\xi\in X$,  
we have, as a consequence of \eqref{balls}, this useful inequality:
\begin{lemma} \label{sum_est}
Let $h=h(X,\mathbb M)$ and let $f:[0,\infty) \to [0,\infty)$ be decreasing and satisfy the following: 
There is a continuous function $g:[0,\infty) \to [0,\infty)$ 
such that $f(xh)\le g(x)$ and that $x^{d-1}g(x)$ is decreasing for $x\ge 1$, 
and is integrable  on $[0,\infty)$. 
Then 
 \begin{equation}\label{decreasing}
 \sum_{\zeta\ne \xi\in X} f({\mathrm{dist}}(\zeta,\xi)) 
 \le 
 \frac{2^dd\,\rho^d\beta_\bbM}{\alpha_\bbM} \int_0^{\infty} g(r) r^{d-1} \rmd r.
 \end{equation}
 \end{lemma}
\begin{proof}
Divide $\bbM$ into $N\approx \rmd_\bbM/h$ annuli $\bsa_n$, 
with center $\xi$ and inner and outer radii $(n-1) h$ 
and $nh$, $n\ge 2$. 
The cardinality of centers in each annulus $\bsa_n$ is approximately 
\[
\# \bsa_n \approx \frac{\mathrm{vol}(B(\xi,nh))-\mathrm{vol}(B(\xi,(n-1)h))}{\mathrm{vol}(B(\xi,q))}.
\]
By \eqref{balls}, we see that  
\[
\#\bsa_n 
\approx 
\frac{\beta_\bbM}{\alpha_\bbM} \frac{(nh)^d-((n-1)h)^d}{q^d}
\le 
\frac{d\beta_\bbM}{q^d\alpha_\bbM} h^d n^{d-1}
=
\frac{d\rho^d \beta_\bbM}{\alpha_\bbM} n^{d-1}
\]
By the assumption that $f$ is decreasing, we have, using $n^{d-1}\le 2^d(n-1)^{d-1}$, $n\ge 2$,  
\begin{eqnarray*}
\sum_{\zeta\in \bsa_n} f({\mathrm{dist}}(\xi,\zeta)) 
&\le& 
\frac{d\rho^d\beta_\bbM}{\alpha_\bbM} f((n-1)h)n^{d-1}\\
&\le& 
\frac{2^d d\,\rho^d\beta_\bbM}{\alpha_\bbM} g((n-1))(n-1)^{d-1}.
\end{eqnarray*}
We have $\sum_{n=2}^Ng((n-1))(n-1)^{d-1} \le \int_0^\infty g(r)r^{d-1} \rmd r$, 
since $g(x)x^{d-1}$ is decreasing. This and the previous inequality then imply \eqref{decreasing}.
\end{proof}

Given $D$  and $X \subset D$, 
we wish to find an extension $\widetilde X \supset X$ so that 
the separation radius is not decreased and the fill distance is controlled.
\begin{lemma}\label{Extension}
Suppose $X \subset D \subset \mathbb M$ has fill distance $h(X,D) = h$ 
and separation radius $q(X) = q$.
Then there is a finite set $\widetilde X$ so that   $\widetilde X\cap D = X$,
$q(\widetilde X) = \min(q,h/2)$
and $h(\widetilde X,\bbM) =  h$.
\end{lemma}
\begin{proof}
We extend $X$ by taking $Z = \bbM\setminus \bigcup_{\xi\in X} B(X,h)$. 
Cover $Z$ by a maximal $\epsilon$-net with $\epsilon = h$ as follows. 

Consider the set of discrete subsets 
$\mathcal{D} = \{D\subset Z\mid h(D,Z) = h, q(D) = h/2\}$. 
This is a partially ordered set under $\subset$ and therefore has a maximal element $D^*$ 
by Zorn's lemma.
This maximal element must satisfy $q(D^*) = h/2$ (since it's in $\mathcal{D}$) and must cover $Z$ 
(if $x\in Z\setminus \bigcup_{z\in D^*}B(z,h)$ then $D^*$ is not maximal). 
It follows that $\widetilde X  = X\cup D^*$ has fill distance $h(\widetilde X , \bbM) = h$ 
and $q(\widetilde X) = \min(q,h/2)$.
\end{proof}
\subsubsection{Sobolev spaces} 
We can define Sobolev spaces in a number of equivalent ways. 
In this article, we focus on $W_p^{\tau}(\Omega)$, where $\tau\in \N$ and  $1\le p<\infty$.
For $p=\infty$, we
make use of the short hand notation (usual for approximation theory) $W_{\infty}^{\tau}  = C^{\tau}$ 
(i.e., substituting the $L_{\infty}$ Sobolev space by the H{\" o}lder space).

Our definition is the one developed in \cite{Aub}, by using the covariant derivative operator. This
permits us to correctly define Sobolev norms and semi-norms on domains.
Namely, 
$$
\|f\|_{W_p^{\tau}(\Omega)} 
= \left(\sum_{k=0}^{\tau}\int_{\Omega} (\langle \nabla^k f,\nabla^k f\rangle_x)^{p/2} \rmd x\right)^{1/p}.
$$
See \cite[Chapter 2]{Aub}, \cite[Section 3]{HNW} or \cite[Chapter 7]{triebel1992} for details.

Here bounded geometry means that $\bbM$ has  a positive injectivity radius 
and that derivatives of the Riemannian metric are bounded
 (see \cite[Section 2]{HNW} for details).
 This fact already implies the Sobolev embedding theorem, 
 as well as a smooth family of local diffeomorphisms (uniform metric isomorphisms),
\cite[(2.6)]{HNW},
 which induce a  family of metric isomorphisms \cite[Lemma 3.2]{HNW} 
 between Sobolev spaces on $\bbM$ 
 and on $\RR^d$. 
\subsection{Sobolev-Mat{\'e}rn kernels}\label{SSS:Kernels}
The kernels we consider in this article are positive definite. 
Much of the theory extends to kernels that are conditionally positive definite; 
for a discussion, see \cite{HNW-p}.

A positive definite kernel  ${ k:\bbM\times \bbM \to \RR}$
satisfies the property that for every finite set $X \subset \bbM$, the collocation matrix 
$$\rmK_{X} :=(k(\xi,\zeta))_{\xi,\zeta\in X}$$ 
is strictly positive definite.

If ${ \tau>d/2}$, then ${ W_2^{\tau}(\bbM)}$ 
is a reproducing kernel Hilbert space, and its kernel is positive definite. 
Conversely, every continuous positive definite kernel is the reproducing kernel for a Hilbert space of 
continuous functions $\calN(k)$ on $\bbM$. 

The positive definite kernels we consider in this article are the Sobolev-Mat{\' e}rn kernels, which  
are the reproducing kernels for the Sobolev space $W_2^m(\bbM)$. 
These were introduced in \cite{HNW}; we will denote them by $\kappa_m$.  
They are also the fundamental solution of the elliptic differential operator, 
 $\mathcal{L} = \sum_{j=0}^m (\nabla^j)^*\nabla^j$ of order $2m$. 
This fact, although not used directly, 
is a key fact used to establish the stationary energy decay estimates 
considered in Section \ref{SS:Lagrange}.

For finite $X\subset \bbM$ we define $S(X) := \mathrm{span}_{\xi\in X}k(\cdot, \xi)$. 
The guaranteed  invertibility of $\rmK_{X}$ is of use in solving interpolation problems -- 
given $\bsy\in \RR^{X}$, one finds $\bsa\in\RR^{X}$ 
so that $\rmK_{X} \bsa = \bsy$. It follows that $\sum_{\xi\in X} a_{\xi} k(\cdot,\xi)$ 
is the unique interpolant to $({\xi},y_{\xi})_{\xi\in X}$ in $S(X)$. 
It is also the case that $\sum_{\xi\in X} a_{\xi} k(\cdot,\xi)$ 
is the interpolant  to $({\xi},y_{\xi})_{\xi\in X}$ with minimum
$\calN(k)$ norm.
%

\section{Lagrange Functions and First Bernstein Inequalities}\label{lag_bernstein}

In this section we introduce the Lagrange functions, which are a localized basis generated by
the kernel $\kappa_m$.
After this we give our first class of Bernstein estimates, 
valid for linear combinations of Lagrange functions.

\subsection{Lagrange functions}\label{SS:Lagrange}
For a positive definite kernel $k$ and a finite 
$X\subset \bbM$, there exists a family of uniquely defined functions 
$(\chi_{\xi})_{\xi\in X} \subset S(X)$ that satisfy $\chi_{\xi}(\zeta) = \delta(\xi,\zeta)$ 
for all $\zeta\in X$, and have the representation $\chi_{\xi} =\sum_{\eta\in X} A_{\eta,\xi} k(\cdot,\eta)$. 
The $\chi_\xi$'s are the \emph{Lagrange functions} associated with $X$; 
they are easily seen to form a basis for $S(X)$.

The $A_{\eta,\xi}$'s can be expressed in a useful way, in terms of an inner product. 
Let $\langle \cdot, \cdot\rangle_{\calN(k)}$ denote the inner product for the reproducing Hilbert space 
$\calN(k)$. 
Because $\chi_\xi\in \calN(k)$, we have that 
$\langle \chi_{\xi} , \kappa_m(\cdot,\eta) \rangle_{\calN(k)}= \chi_\xi(\eta)$. 
Representing a second  $\chi_\zeta$ by $\chi_{\zeta} =\sum_{\eta\in X}  A_{\eta,\zeta} k(\cdot,\eta)$, 
we obtain
\begin{equation}\label{inner_product_coefficients}
\big\langle \chi_{\xi} , \chi_{\zeta} \big\rangle_{\calN(k)}
=\big\langle \chi_{\xi} ,  \sum_{\eta\in X}A_{\eta,\zeta} k(\cdot,\eta) \big\rangle_{\calN(k)} 
= \sum_{\eta \in X}A_{\eta,\zeta}\chi_{\xi}(\eta)  = A_{\xi,\zeta}.
\end{equation}
If $k=\kappa_m:\bbM\times \bbM \to \RR$ is a Sobolev-Mat\'ern kernel, 
then, by virtue of $\kappa_m$ being a reproducing kernel for $\calN(\kappa_m)\approx W_2^m(\bbM)$, 
we can make the following ``bump estimate'' on the $A_{\eta,\xi}$'s . 
Consider a $C^\infty$ function $\psi_{\xi,q}:\bbM\to [0,1]$ 
that is compactly supported in $B(\xi,q)$ and that
 satisfies $\psi_{\xi,q}(\xi)=1$. 
 Moreover, the condition on  $\mathrm{supp}(\psi_{\xi,q})$ implies that, 
 for any $\zeta\in X$, $\psi_{\xi,q}(\zeta)=\delta(\xi,\zeta)$. 
 Because $\psi_{\xi,q}\in W_2^m\approx \calN(\kappa_m)$ and $\chi_{\xi}$ 
 is the minimum norm interpolant 
 to $\zeta\to \delta(\xi,\zeta)$, we have that 
$$
\|\chi_{\xi}\|_{\calN(\kappa_m)}
\le
\|\psi_{\xi,q}\|_{\calN(\kappa_m)}
\le
C \|\psi_{\xi,q}\|_{W_2^{m}(\bbM)}
 \le Cq^{\frac{d}2-m}.
 $$
As a consequence, the $A_{\xi,\zeta}$'s are uniformly bounded: 
\begin{equation}\label{basiccoeffs}
|A_{\xi,\zeta}| = |\langle \chi_{\xi},\chi_{\zeta}\rangle_{\calN(\kappa_m)}|\le C q^{d-2m}.
\end{equation}
This bound is rather rough, and can be substantially improved. 
In fact, when $X$ is sufficiently dense in $\bbM$, there exist constants $C$, and $\nu>0$, 
which depend on $\kappa_m$ (see\cite{FHNWW}),
so that the coefficient bound
\begin{equation}\label{coeff}
|A_{\xi,\zeta}| 
= 
|\langle \chi_{\xi},\chi_{\zeta}\rangle_{\calN(\kappa_m)}|
\le 
C q^{d-2m} \mathrm{exp}\left(-{\nu} \frac{{\mathrm{dist}}(\xi,\zeta)}{h}\right)
\end{equation}
holds. The proof of this estimate is, \emph{mutatis mutandis}, 
the same as that for \cite[Eqn.~5.6]{FHNWW}.

Under the same hypotheses,
 we have the   spatial decay of the Lagrange function:
\begin{equation}
|\chi_{\xi}(x)| \le C \rho^{m-d/2} \mathrm{exp}\left(-\mu \frac{{\mathrm{dist}}(x,\xi)}{h}\right), \label{ptwise}
\end{equation}
with
$\mu = 2\nu $.
 Both (\ref{ptwise}) and (\ref{coeff}) are consequences of
the zeros estimate \cite[(A.15)]{HNW-p} on $\bbM$
and a more basic estimate,
\begin{equation}
\|\chi_{\xi} \|_{W_2^m(\bbM \setminus B(\xi,R))} 
\le 
C q^{d/2 - m} \mathrm{exp}\left(-\mu \frac{R}{h}\right)\label{energy}
\end{equation}
which we call an energy estimate.
When (\ref{energy}) holds for a system of Lagrange functions, 
we say it exhibits {\em stationary exponential decay of order $m$}
 Stationary decay of order $m$ was demonstrated for Lagrange functions generated by
  Sobolev-Mat{\' e}rn kernels on compact Riemannian manifolds in \cite{HNW}. 
 (Specifically, these results are found in \cite[Corollary 4.4]{HNW} for (\ref{energy}) and 
 in \cite[Proposition 4.5]{HNW} for (\ref{ptwise}).) 
Similar bounds hold for Lagrange functions associated with other kernels, 
both positive definite and conditionally positive definite, as discussed in \cite{HNW-p} and \cite{HNRW2}.

We stress
that
to get estimates (\ref{energy}), (\ref{ptwise}) and (\ref{coeff}), 
the point set $X$ must be dense in $\bbM$. 
This is clearly problematic when we consider behavior over $\Omega\subsetneq \bbM$ 
and $X\subset\Omega$ 
(which is a focus of this article).
To handle this, for a given point set we require the dense, quasi-uniform extension to 
$\mathbb M$ that was developed in  Lemma~\ref{Extension}.

\subsection{Bernstein type estimates for (full) Lagrange functions}\label{full_bernstein}

We develop partial Bernstein inequalities for functions 
$s = \sum_{\xi\in X}a_\xi \chi_{\xi}$.
Our goal is to control Sobolev norms
$\|s\|_{W_2^{\sigma}}$ by the $\ell_2(X)$ norm on the coefficients: $\|\bsa\|_{\ell_2(X)}$.
We have the following theorem.
\begin{theorem}\label{synthesis_thm} 
If $X$ is sufficiently dense in $\Omega$ and $0\le \sigma \le m$, then there exists $C<\infty\ $ such that 
\begin{equation}\label{Synthesis}
\big\|\sum_{\xi\in X} a_{\xi} \chi_{\xi}\big\|_{W_2^{\sigma}(\Omega)} 
\le 
C \rho^{m} h^{d/2 -\sigma } \big\|\bsa \big\|_{\ell_2(X)}.
\end{equation} 
\end{theorem}
\begin{proof}
Since $\Omega\subseteq \bbM$ and $W_2^{\sigma}(\Omega) \subseteq W_2^{\sigma}(\mathbb M) $, 
we only need to prove the result for $\Omega=\mathbb M$. In addition, we can replace $X$ with 
$\widetilde X$, the extension of $X$ to $\mathbb M$, whose existence was shown in 
Lemma~\ref{Extension}. The point is that once the result is shown true for $\widetilde X$, 
we just restrict $a_\xi$'s to $\xi \in X$, setting $a_\xi=0$ for $\xi \in \widetilde X \setminus X$. 

To begin, we use (\ref{coeff}) to observe that $\chi_{\xi} \in W_2^m(\bbM)$, whence we obtain
\begin{eqnarray*}
\big\| \sum_{\xi\in \widetilde X} a_{\xi} \chi_{\xi}\big\|_{W_2^m(\bbM)}^2
&\le&
C \big\| \sum_{\xi\in \widetilde X} a_{\xi} \chi_{\xi}\big\|_{\calN(\kappa_m)}^2\\
&=&
C \sum_{\xi\in \widetilde X} 
\sum_{\zeta\in \widetilde X}|a_{\xi}||a_{\zeta}| 
\bigl|\langle  \chi_{\xi},\chi_{\zeta} \rangle_{\calN(\kappa_m)}\bigr|\\
&\le&
C q^{d-2m}\sum_{\xi\in \widetilde X} 
\sum_{\zeta\in \widetilde X}|a_{\xi}||a_{\zeta}| e^{-\nu \frac{{\mathrm{dist}}(\xi,\zeta)}{h}}
\end{eqnarray*}
We can split this into diagonal and off-diagonal parts, leaving
$$
\big\| \sum_{\xi\in \widetilde X} a_{\xi} \chi_{\xi}\big\|_{W_2^m(\bbM)}^2
\le
C q^{d-2m}
\bigg(\sum_{\xi\in \widetilde X}|a_{\xi}|^2+
\sum_{\xi\in \widetilde X} \sum_{\ \zeta\in \widetilde X,\zeta\ne \xi}|a_{\xi}||a_{\zeta}| 
e^{-\nu \frac{{\mathrm{dist}}(\xi,\zeta)}{h}}\bigg).
$$

From this we have
$
\|
\sum_{\xi\in \widetilde X} a_{\xi} \chi_{\xi}\|_{W_2^m(\bbM)}
\le
C q^{d/2 - m} \left(\| \bsa\|_{\ell_2(\widetilde X)} + (II)^{1/2}\right)
$.
We focus on the off-diagonal part $II$. Since each term appears twice,  we can make the estimate
\begin{eqnarray*}
\sum_{\xi\in \widetilde X}\sum_{\zeta\ne \xi}
|a_{\xi}||a_{\zeta}|e^{-\nu \frac{{\mathrm{dist}}(\xi,\zeta)}{h}} 
&\le& 
\sum_{\xi\in \widetilde X}\sum_{\ \zeta\in \widetilde X,\zeta\ne \xi}
|a_{\xi}|^2 e^{-\nu \frac{{\mathrm{dist}}(\xi,\zeta)}{h}} \\
&\le& 
C \rho^d\left( \int_0^\infty e^{-\nu r} r^{d-1} \rmd r\right) \sum_{ \xi\in \widetilde X}|a_{\xi}|^2  .
\end{eqnarray*}
The first inequality uses the estimate $|a_{\xi}||a_{\zeta}| \le\frac12 ( |a_{\xi}|^2 + |a_{\zeta}|^2)$.
The second inequality follows from (\ref{decreasing}). We have demonstrated that
\begin{equation}\label{Synthesis_Sobolev}
\big\|\sum_{\xi\in \widetilde X} a_{\xi} \chi_{\xi}\big\|_{W_2^m(\bbM)} 
\le 
C \rho^{d/2} q^{d/2 -m} \|\bsa \|_{\ell_2(\widetilde X)}
\le 
C \rho^{m} h^{d/2 -m} \left\|\bsa \right\|_{\ell_2(\widetilde X)}. 
\end{equation}

On the other hand, using (\ref{ptwise}) we have
\begin{eqnarray*}
\big\|\sum_{\xi\in \widetilde X} a_{\xi} \chi_{\xi}\big\|_{L_2(\bbM)}^2 
&\le&  
\sum_{\xi\in \widetilde X} 
  \sum_{\zeta\ne \xi}
      |a_{\xi}||a_{\zeta}| |\langle  \chi_{\xi},\chi_{\zeta} \rangle_2|\\
&\le& 
C \rho^{2m-d} \sum_{\xi\in \widetilde X} 
  \sum_{\substack{ \zeta\in \widetilde X\\ \zeta\ne \xi}}
      |a_{\xi}||a_{\zeta}| 
          \int_{\bbM} 
              e^{-2\nu \frac{{\mathrm{dist}}(x,\xi)}{h}} e^{-2\nu \frac{{\mathrm{dist}}(x,\zeta)}{h}}
          \rmd x .
\end{eqnarray*}
The integral can be estimated over two disjoint regions 
(the part of $\bbM$ closer to $\xi$ and the part closer to $\zeta$) 
to obtain
\begin{eqnarray*}
\big\|\sum_{\xi\in \widetilde X} a_{\xi} \chi_{\xi}\big\|_{L_2(\mathbb M)}^2
&\le& 
C\rho^{2m-d} h^d 
  \sum_{\xi\in \widetilde X} 
  \sum_{\ \zeta\in \widetilde X,\zeta\ne \xi}
      |a_{\xi}||a_{\zeta}| e^{- {\nu} \frac{{\mathrm{dist}}(\xi,\zeta)}{h}}\\
&\le & 
C\rho^{2m}  h^d \sum_{ \xi\in \widetilde X}|a_{\xi}|^2 .
\end{eqnarray*}
The second inequality repeats the estimate used on 
$\big\|\sum_{\xi\in \widetilde X} a_{\xi} \chi_{\xi}\big\|_{W_2^m(\bbM)}$.
It follows that 
\begin{equation}\label{Synthesis_L_2}
\big\|\sum_{\xi\in \widetilde X} a_{\xi} \chi_{\xi}\big\|_{L_2(\bbM)} 
\le 
C \rho^{m} h^{d/2 } \big\|\bsa \big\|_{\ell_2(\widetilde X)}.
\end{equation}

Define the operator 
$V:\ell_2(\widetilde X)\to W_2^m(\bbM):\bsa \mapsto \sum_{\xi\in \widetilde X}a_{\xi} \chi_{\xi}$. 
We interpolate between (\ref{Synthesis_Sobolev}) and
(\ref{Synthesis_L_2}), using the fact that 
$W_2^{\sigma}(\bbM) = B_{2,2}^{\sigma}(\bbM) = [L_2(\bbM),W_2^m(\bbM)]_{\tfrac{\sigma}{m},2}$ 
(cf. \cite{triebel1992}). 
As noted at the start, this implies the result for $\Omega \subseteq \mathbb M$ and $X\subset \Omega$.
\end{proof}
%

\section{Local Lagrange Functions}\label{local_lagrange}
We now consider locally constructed basis functions. 
We employ a small set of centers from $X$ to construct ``local'' Lagrange functions: 
For each $\xi\in X$, we define 
\[
\Upsilon(\xi) := \{\zeta\in X\mid{\mathrm{dist}}(\zeta,\xi)\le Kh |\log h|\}, 
\]
where $K>0$ is a parameter used to adjust the number of points in $\Upsilon(\xi)$.

We define the \emph{local Lagrange function} $b_\xi$ at $\xi$ to be the Lagrange function for 
$\Upsilon(\xi)$.  
We will call $\Upsilon(\xi)$ the {\em footprint} of $b_\xi$. 
Of course, $b_\xi\in S(\Upsilon(\xi))$. The choice of the parameter $K$ depends on the 
constants appearing in  the  stationary exponential decay (\ref{energy}), the conditions we place on
the manifold $\bbM$ and the rate at which we wish $b_{\xi}$ to have decay away from $\xi$.

$K$ may be chosen so that for a prescribed $J$, which depends linearly on $K$ and other parameters, 
we can ensure that  $\|\chi_{\xi} - b_{\xi}\|_{L_\infty(\mathbb M)} =\mathcal{O}(h^J)$ holds. 
(See \eqref{Sob_loc_full_error}.)

The main goal of this section is to provide Sobolev estimates 
on the difference between locally constructed 
functions $b_{\xi}$ and the analogous (full Lagrange) functions $\chi_{\xi}$. 
As in \cite{FHNWW} the analysis of this new basis is considered in two steps. 
First, an intermediate basis function
${\widetilde \chi}_{\xi}$ is constructed and studied: 
the {\em truncated Lagrange function}. 
 These functions employ the same footprint as $b_{\xi}$ 
 (i.e., they are members of $S(\Upsilon(\xi))$)
but their construction is global rather than local. 
This topic  is considered in Section \ref{SS:TLF}. 
Then, a comparison is made between the truncated Lagrange function and the local Lagrange function. 
In Section~\ref{SS:LLF}, 
we will show that the error between local and truncated Lagrange functions is controlled 
by the size of the  coefficients  
in the expansion of
 $b_{\xi} -{\widetilde \chi}_{\xi}$ in the standard (kernel) basis for $S(\Upsilon(\xi))$.

\subsection{Truncated Lagrange functions}\label{SS:TLF}

For a (full) Lagrange function 
$\chi_{\xi} = \sum_{\zeta\in X} A_{\xi,\zeta} k(\cdot, \zeta)  \in S( X)$
on the point set $ X$, the truncated Lagrange function 
${\widetilde \chi}_\xi =  \sum_{\zeta\in\Upsilon(\xi)} A_{\xi,\zeta} k(\cdot, \zeta)$ 
is a function in $S(\Upsilon(\xi))$ obtained by removing the $A_{\xi,\zeta}$'s for $\zeta$ 
not in $\Upsilon(\xi)$. 
The cost of truncating can be measured using the norm of the omitted coefficients (the tail).

\begin{lemma} \label{truncated_sum}
Let $\bbM$ be as in Section~\ref{SS:manifolds} 
and let $\kappa_m$ be a Sobolev-Mat\'ern kernel generating  
$\{\chi_\xi: \xi\in X\}$. 
Suppose $ X\subset \bbM$ has fill distance $0<h\le h_0$ and separation radius $q>0$. 
Let $K> (4m-2d)/\nu$ and for each $\xi\in X$, 
let $\Upsilon(\xi) = \{\zeta\in  X\mid {\mathrm{dist}}(\xi,\zeta) \le K h|\log h|\}$. 
Then
$$
\sum_{\zeta\in  X\setminus \Upsilon(\xi)} |A_{\xi,\zeta}|
\le 
C \rho^{2m} h^{K\nu/2 +d-2m}.
$$
\end{lemma}
\begin{proof}
The inequality (\ref{coeff}) guarantees  that 
\begin{eqnarray*}
\sum_{\zeta \in  X \setminus \Upsilon(\xi)} |A_{\xi,\zeta}|&\le& 
Cq^{d-2m} 
\sum_{{\mathrm{dist}}(\zeta,\xi)\ge Kh|\log h|} \exp\left(-\nu\frac{{\mathrm{dist}}(\xi,\zeta)}{h} \right)\\
&\le& 
Cq^{-2m}
\int_{y\in \bbM \setminus B(\xi,Kh|\log h|)}
 \exp\left(-\nu\frac{{\mathrm{dist}}(\xi,y)}{h} \right)\rmd y\\
&\le& 
Cq^{-2m}
\int_{Kh|\log h|}^{\infty}
 \exp\left(-\nu\frac{r}{h} \right)r^{d-1}\rmd r.
\end{eqnarray*}
A simple way\footnote{The integral can be done exactly. 
However, we don't need to do that here.} 
to estimate this involves splitting $\nu = \nu/2 +\nu/2$ and writing
\begin{multline*}
\sum_{\zeta \in  X \setminus \Upsilon(\xi)} |A_{\xi,\zeta}|\\
\le
C h^d q^{-2m}
\left(\int_{K|\log h|}^{\infty} r^{d-1}
 \exp\left(- K |\log h| \frac{\nu}{2} \right) 
 \exp\left(- r \frac{\nu}{2} \right)\rmd r\right)\\
 \le
 C h^d q^{-2m}
 h^{K\nu/2}.
\end{multline*}
The lemma follows.
\end{proof}

Standard properties of reproducing Hilbert kernels imply that, 
because $\mathbb M$ is a compact metric space, $\kappa_m(x,y)$ 
is continuous on $\mathbb M \times \mathbb M$. 
Consequently, $\kappa_m(x,x) = \|\kappa_m(\cdot,x)\|_{\mathcal N(\kappa_m)}^2$ 
is uniformly bounded in $x$. 
Moreover, since $\mathcal N(\kappa_m)$ and $W_2^m(\mathbb M)$ are norm equivalent, 
there is a constant $\Gamma$ such  that 
\[
\sup_{x\in \mathbb M}\|\kappa_m(\cdot,x)\|_{W_2^m(\mathbb M)} 
\le C\sup_{x\in \mathbb M}\|\kappa_m(\cdot,x)\|_{\calN(\kappa_m)}
\le \Gamma_{\kappa_m}.
\] 
From Lemma~\ref{truncated_sum} and the inequality above, we have that
\begin{equation}\label{general_norm}
\|\chi_{\xi} - {\widetilde \chi}_{\xi}\|_{W_2^m(\mathbb M)}
\le
\Gamma_{\kappa_m} \sum_{\zeta \in  X \setminus \Upsilon(\xi)} |A_{\xi,\zeta}|  
\le 
C \Gamma_{\kappa_m}  
 \rho^{2m}h^{K\nu/2-2m+d}.
\end{equation}
 Applying the Sobolev embedding theorem then yields the result below.
  \begin{proposition} \label{SobMat}
  Let $\kappa_m$ the  Sobolev-Mat{\' e}rn kernel , with $m>d/2$. Then, if $1\le p<\infty$ and 
  $\sigma \le m- (\frac{d}{2} - \frac{d}{p})_+$,   
  or if $p=\infty$ and $0\le \sigma < m-d/2$,
  we have 
 \begin{equation} \label{p_sigma__bnd}
\|\chi_{\xi} - {\widetilde \chi}_{\xi} \|_{W_p^\sigma(\bbM)} 
 \le 
C\Gamma_{\kappa_m}\rho^{2m}h^{K\nu/2 +d -2m}, \quad C=C_{\sigma,m,p}.
\end{equation}
 In particular, if $p=\infty$ and $\sigma=0$, we have
 \begin{equation} \label{infty_bnd}
 \|\chi_{\xi} - {\widetilde \chi}_{\xi} \|_{L_\infty(\mathbb M)} 
 \le C_m\Gamma_{\kappa_m}\rho^{2m}h^{K\nu/2 +d -2m}.
 \end{equation}
\end{proposition}
\begin{proof}
This follows from (\ref{general_norm}) by applying the Sobolev embedding theorem to 
$\|\chi_{\xi} - {\widetilde \chi}_{\xi} \|_{W_p^\sigma(\bbM)}$. 
\end{proof}

\subsection{Local Lagrange Function Distance Estimates}\label{SS:LLF}
In this section, we consider bounding the distance between $b_{\xi}$ and $\chi_{\xi}$ and also 
$b_{\xi}$ and $\widetilde \chi_{\xi}$, using Sobolev norms. 
The argument we will use is essentially the one used on the sphere in \cite{FHNWW}.

By construction, both $b_\xi$ and ${\widetilde \chi}_{\xi}$ are in $S(\Upsilon(\xi))$, and thus 
$b_{\xi} - {\widetilde \chi}_{\xi} \in S(\Upsilon(\xi))$ is, too. 
Hence, $b_{\xi}- {\widetilde \chi}_{\xi} = \sum_{\zeta\in \Upsilon(\xi)} a_{\zeta} \kappa_m(\cdot,\zeta)$. 
Let $\bsa:=(a_\zeta)_{\zeta \in \Upsilon(\xi)}$ and $\bsy=(b_\xi-{\widetilde \chi}_{\xi})|_{\Upsilon(\xi)}$ 
where $\bsa$ and $\bsy$ are related by $\rmK_{\Upsilon(\xi)} \bsa = \bsy$. 

We can write $\bsy$ another way. 
Since $b_\xi$ is a Lagrange function for $\Upsilon(\xi)$, 
we have that $b_\xi(\zeta) = \delta_{\xi,\zeta}$ when $\zeta\in \Upsilon(\xi)$. 
However, because $\chi_\xi$ is a Lagrange function for all $X$, it also satisfies 
$\chi_\xi(\zeta) = \delta_{\xi,\zeta}$, $\zeta\in \Upsilon(\xi)$. 
Consequently,  $\bsy=(\chi_\xi-{\widetilde \chi}_{\xi})|_{\Upsilon(\xi)}$. 

Using this form of $\bsy$ we have that 
$\|\bsy\|_1
\le 
(\# \Upsilon(\xi)) \|\bsy\|_\infty\le (\# \Upsilon(\xi)) \| \chi_{\xi} - {\widetilde \chi}_{\xi}\|_{L_\infty(\bbM)}
$. 
From \eqref{infty_bnd} and the bound $\# \Upsilon(\xi)\le C \rho^d |\log h|^d$, we arrive at 
\[
\|\bsy\|_1\le C\rho^{2m+d}h^{K\nu/2 +d -2m}  |\log h|^d.
\]

The matrix $(\rmK_{\Upsilon(\xi)})^{-1}$ has entries 
$(A_{\zeta,\eta} )_{\zeta,\eta\in\Upsilon(\xi)}$. \ 
These  can be estimated by (\ref{basiccoeffs}): 
 $|A_{\zeta,\eta}| \le  Cq^{d-2m}$. \ 
 It follows that $(\rmK_{\Upsilon(\xi)})^{-1} $ has $\ell_1$ matrix norm 
 $$
 \big\|\left(\rmK_{\Upsilon(\xi)}\right)^{-1}\big\|_{1\to 1} 
 \le 
 C (\# \Upsilon(\xi)) q^{d-2m}
 \le 
 C  \rho^{2m} |\log h|^d h^{d-2m}.
 $$
 This and the bound on $\| \bsy\|_1$ above imply that
\begin{equation} \label{norm_a_bnd}
\|\bsa\|_1
\le 
\big\|\left(\rmK_{\Upsilon(\xi)}\right)^{-1}\big\|_{1\to 1}
\|\bsy\|_1
\le
 C  \rho^{4m+d} |\log h|^{2d} h^{K\nu/2 +2d -4m}. 
\end{equation}
 
Under the conditions in Proposition~\ref{SobMat}, $b_\xi - \widetilde \chi_\xi$ is in 
$W_p^\sigma(\bbM)$, as is each $\kappa_m(\cdot,\zeta)$. 
Consequently, 
$
\|b_{\xi}- {\widetilde \chi}_{\xi}\|_{W_p^\sigma(\bbM)} 
\le 
\|\bsa\|_1 \max_{z \in \bbM} \|\kappa_m(\cdot,z)\|_{W_p^\sigma(\bbM)}\le \Gamma_{\kappa_m}
 \| \bsa\|_1
 $. 
Using the triangle inequality, the bound in \eqref{norm_a_bnd}, 
and the estimate above, we have the following result:
\begin{lemma}\label{compact_error}
Let $\bbM$ be as in Section~\ref{SS:manifolds} and let $\kappa_m$ be a Sobolev-Mat\'ern kernel. 
Then, we have,  
for $0\le \sigma \le m -(d/2-d/p)_+$,  or with $p=\infty$ and $0\le \sigma < m-d/2$,
\begin{equation}
\left\|b_{\xi} -\chi_{\xi}\right \|_{W_p^{\sigma}(\bbM)}
\le
C\Gamma_{\kappa_m}  \rho^{4m+d}h^{K\nu/2 +2d -4m}|\log h|^{2d}, \ C=C_{m,p,\sigma}
\end{equation}
\end{lemma}

We remark that $|\log h|^{2d} \le C h^{-1}$, so that either by finding a sufficiently small $h^*$, 
so that this holds for $h<h^*$, or by increasing the constant, or both we have
\begin{equation}\label{Sob_loc_full_error}
\left\|b_{\xi} -\chi_{\xi}\right \|_{W_p^{\sigma}(\bbM)}
\le C  
\rho^{4m+d}h^{K\nu/2 +2d -4m-1}.
\end{equation}

\subsection{Bernstein 
estimate for local Lagrange functions}\label{local_bernstein}

In this section we discuss the  local Lagrange functions $b_\xi$ generated by $\kappa_m$ 
and the centers $ X$.
We develop partial Bernstein inequalities, 
where for functions of the form  $s= \sum_{\xi\in  X}a_\xi b_{\xi}$ the norms 
$\|s\|_{W_2^{\sigma}}$ are controlled by an $\ell_2$ norm on the coefficients: $\|\bsa\|_{\ell_2( X)}$.

We will now obtain estimates similar to (\ref{Synthesis}) for the expansion 
$\sum_{\xi\in  X} a_{\xi} b_{\xi}$. 
In contrast to the full Lagrange basis, which is globally decaying, 
we have a family of functions $(b_{\xi})_{\xi\in X}$ 
whose members are uniformly small (on compact sets), 
but do not necessarily decay (at least not in a stationary way).

\begin{theorem}\label{main_local_bernstein}
Suppose  $X$ is sufficiently dense in $\Omega$. 
Assume $K\nu +d -4m -1 \ge d/2 - \sigma$. 
Then there is $C$, depending on the constants appearing in \eqref{balls} 
 and \eqref{ptwise} 
 so that
\begin{equation}\label{first_inverse_estimate}
\big\|\sum_{\xi\in  X} a_{\xi} b_{\xi}\big\|_{W_2^{\sigma}(\Omega)} 
\le 
C_\bbM \rho^{4m+2d} h^{d/2 - \sigma} \|\bsa\|_{\ell_2( X)}.
\end{equation}
\end{theorem}
\begin{proof}
As in the case of Theorem~\ref{synthesis_thm}, because $\Omega \subseteq \bbM$, 
we only have to prove the result for $\bbM$.  
We start with the basic splitting 
$$
s := \sum_{\xi\in  X} a_{\xi} b_{\xi} 
=\big(  \sum_{\xi\in  X} a_{\xi} \chi_{\xi}\big)+\big(\sum_{\xi\in  X} a_{\xi}( b_{\xi}-\chi_{\xi} )\big) 
=: G+B .$$ 
Applying the Sobolev norm gives
$\|s\|_{W_2^{\sigma}(\bbM)}^2 
\le 
\|G\|_{W_2^{\sigma}(\bbM)}^2 +\|B\|_{W_2^{\sigma}(\bbM)}^2$. 
From (\ref{Synthesis}), we have 
$\|G\|_{W_2^{\sigma}(\bbM)} \le C\rho^{m} h^{d/2-\sigma} \|\bsa\|_{\ell_2( X)}$.

We now restrict our focus to $B$. 
For $|\alpha|\le m$,
H{\"o}lder's inequality ensures that
$
\| \sum_{\xi\in  X} a_{\xi} \nabla^{\alpha} ( b_{\xi}-\chi_{\xi} ) \|_x
\le
\big(\sum_{\xi\in  X} |a_{\xi}|^2 \big)^{1/2}  
\big(\sum_{\xi\in  X} \| \nabla^{\alpha}(b_{\xi}-\chi_{\xi})\|_x^2 \big)^{1/2}
$
holds.
Here we have used, 
for a rank $\alpha$-covariant  tensor field $F$ 
(i.e., a smooth section of the vector bundle of rank $\alpha$ covariant tensors),  
the norm on the fiber at $x$ given by the Riemannian metric, i.e., 
$\|F\|_x$ is the norm of the tensor $F(x)$. 

Therefore, for $0\le \sigma\le m$,
\begin{eqnarray*}
\|B\|_{W_2^{\sigma}(\bbM)}
&\le& 
\|\bsa\|_{\ell_2( X)}\big\|  \sum_{\xi\in X} (b_{\xi}-\chi_{\xi})\big\|_{W_2^{\sigma}(\bbM)}\\
&\le& 
\|\bsa\|_{\ell_2( X)}  \sum_{\xi\in X}\left\|(b_{\xi}-\chi_{\xi})\right\|_{W_2^{\sigma}(\bbM)}\\
&\le&
\|\bsa\|_{\ell_2( X)} 
(\#X)
\max_{\xi\in X}\left\|(b_{\xi}-\chi_{\xi})\right\|_{W_2^{\sigma}(\bbM)}
\end{eqnarray*}
The inequality $\|B\|_{W_2^{\sigma}(\bbM)} \le C \rho^{4m+2d} h^{K\nu/2 +d -4m-1} \|\bsa\|_{\ell_2( X)}$
follows by 
applying Lemma \ref{compact_error}, and the fact that $\#X\le C \rho^d h^{-d}$.
Inequality
\eqref{first_inverse_estimate}
follows, which completes the proof.  
\end{proof}
%

\section{Stability Results and Inverse Inequalities}\label{Cond}
In this section we consider finite dimensional spaces 
$V_{\Xi} = \mathrm{span}_{\xi \in \Xi} \chi_{\xi}$ 
and ${\widetilde V}_{\Xi} = \mathrm{span}_{\xi \in \Xi} b_{\xi}$, 
using the Lagrange and local Lagrange functions considered in 
Sections \ref{full_bernstein} and \ref{local_bernstein}.
We note that the localized functions $\chi_{\xi}$ and $b_{\xi}$ 
are indexed by a dense set of centers $X\subset \bbM$,
but the spaces $V_{\Xi}$ and ${\widetilde V}_{\Xi}$ are constructed using 
a restricted set of centers $\Xi =  X\cap \Omega$, 
corresponding to the centers located inside $\Omega\subset \bbM$, 
which the underlying region over which we take the $L_2$ norm.

\subsection{The domain $\Omega$ }\label{bounded_domain}
We now consider a compact region $\Omega \subset \bbM$. 
This presents two challenges. 

The first concerns the density of point sets $\Xi \subset \Omega$. 
Unless $\Omega = \bbM$, the given set $\Xi$ does not itself satisfy the density condition 
$h(\Xi,\bbM)<h_0$. 
For this, we need a larger set $X\subset \bbM$ with points lying outside of $\Omega$ 
(in fact, when working with local Lagrange functions $b_{\xi}$, it suffices to consider 
$X\subset \{x\in \bbM\mid {\mathrm{dist}}(x,\Omega)<K h |\log h|\}$).
This assumption is in place to guarantee decay of the basis functions.
It would be quite reasonable to be ``given'' initially only the set $\Xi\subset \Omega$  
and to use this to construct $X$. Lemma \ref{Extension} demonstrates that it is possible to extend a 
given set of centers $X\subset \Omega$  in a controlled way to obtain a dense subset of $\bbM$.

The second challenge concerns the domain $\Omega$.
Previously we have not needed to make extra assumptions
about such a region, but for estimates relating $\|\bsa \|_{\ell_2}$ and  
$\|\sum_{\xi} a_{\xi} b_\xi\|_{L_2(\mathbb M)}$ or $\|\sum_{\xi} a_{\xi} \chi_\xi\|_{L_2(\mathbb M)}$, 
the boundary becomes slightly more important.
Fortunately, the extra assumption we make on $\Omega$ is quite mild -- 
it is given below in Assumption \ref{Man}.

For the remainder of the article, we assume 
$\Omega\subset \bbM$ satisfies the Boundary Regularity condition and
$\Xi\subset \Omega$ is finite.  
We utilize the extended point set ${\widetilde \Xi}$ from 
Lemma \ref{Extension}; 
this gives rise to the family $(\chi_{\xi})_{\xi\in {\widetilde \Xi}}$.
With this setup, we define 
$$
V_{\Xi} 
:= 
\mathrm{span}_{\xi\in\Xi} \chi_{\xi}\ (\text{\rm Full Lagrange})
\ 
\text{and}
\ 
{\widetilde V}_{\Xi} := \mathrm{span}_{\xi\in\Xi} b_{\xi}\ (\text{Local Lagrange}).
$$
Note that $V_{\Xi} \subset S({\widetilde \Xi})$, while 
${\widetilde V}_{\Xi} \subset S({\widetilde \Xi} \cap \{x\in \bbM\mid {\mathrm{dist}}(x,\Omega) 
\le 
K h|\log h|\}) \subset S({\widetilde \Xi}).$ 
A property of $\bbM$, in force throughout the article, is the following.
\begin{assumption}[Boundary Regularity]\label{Man} 
There exists a constant 
$0<\alpha_{\Omega}$ for which the 
following holds:
for all $x\in \Omega$  and all $r\le \rmd_\bbM$ , 
$$\alpha_{\Omega} r^d \le \mathrm{vol}(B(x,r)\cap \Omega).$$
\end{assumption}
Note that this holds when $\Omega$ satisfies an interior cone condition.
%

\subsection{Stability 
on $\Omega$}\label{total_stability}
In this section we show that the synthesis operators $\bsa\mapsto \sum_{\xi\in\Xi} a_{\xi} \chi_{\xi}$
and $\bsa\mapsto \sum_{\xi\in\Xi} a_{\xi} b_{\xi}$ are bounded above and below from $\ell_p(\Xi)$ to $L_p(\Omega)$.

In addition to the pointwise and coefficient decay (namely  (\ref{ptwise}) and (\ref{coeff})) 
stemming from (\ref{energy}), 
we can employ the following uniform equicontinuity property of the Lagrange functions.
There is $0<\epsilon \le 1$ so  that 
\begin{equation}\label{equicontinuity} 
| \chi_{\xi}(x) - \chi_{\xi}(y)|
\le 
C \left[\frac{{\mathrm{dist}}(x,y)}{q}\right]^{\epsilon} 
\end{equation}
with constant $C$ depending only on $\epsilon$,  the mesh ratio $\rho = h/q$, 
and the constants in (\ref{energy}).
This follows from the energy estimate (\ref{energy})
and a zeros estimate \cite[Corollary A.15]{HNW-p},
and the embedding $C^{\epsilon}(\bbM) \subset W_2^m(\bbM)$ where $0<\epsilon< m-d/2$.
We refer the interested reader to  \cite[Lemma 7.2]{HNSW} for details.
%

%
%
\begin{proposition}\label{lowercomparison} 
Let $\Omega\subseteq\bbM$ be a compact domain satisfying Assumption \ref{Man}. 
Then for the Lagrange functions corresponding to $\kappa_m$, 
there exist constants $c,C>0$ and $q_0>0$, so that for $q<q_0$, 
for $1\le p\le \infty$ and  for all functions in $V_{\Xi}$,
\begin{equation}\label{low_comp_eqn}
c \left\| \bsa  \right \|_{\ell_p(\Xi)}       
\le  
q^{-d/p} \|\sum_{\xi\in\Xi}a_{\xi}\chi_{\xi}\|_{L_p(\Omega)} \le C\left\| \bsa  \right \|_{\ell_p(\Xi)} .
\end{equation}
If, in addition $K\nu/2+2d -4m-2 =:\varepsilon>0$, with $K$ chosen sufficiently large, then
\begin{equation}\label{local_low_comp_eqn}
\frac{c}{2} \left\| \bsa  \right \|_{\ell_p(\Xi)}       
\le  
q^{-d/p} \|\sum_{\xi\in\Xi}a_{\xi}b_{\xi}\|_{L_p(\Omega)} \le \frac{3C}{2}\left\| \bsa  \right \|_{\ell_p(\Xi)}.
\end{equation}

\end{proposition}
\begin{proof}
We begin with the case in which $\Omega=\mathbb M$ and $s=\sum_{\xi\in\Xi}a_{\xi}\chi_{\xi}\in V_{\Xi}$. Then \eqref{low_comp_eqn} follows directly from \cite[Proposition 3.10]{HNSW}. 
In particular, we note that the boundary regularity assumption guarantees that $\mathbb M$
satisfies \cite[Assumption 2.1]{HNSW}.  
The family of functions $(\chi_{\xi})_{\xi\in\Xi}$ fulfills the three requirements on
$(v_{\xi})_{\xi\in\Xi}$. 
\begin{enumerate}
\item They are Lagrange functions on $\Xi$ (this is \cite[Assumption 3.3]{HNSW}),
\item The decay property given in (\ref{ptwise}) guarantees that \cite[Assumption 3.4]{HNSW} holds 
(with $\mathrm{r}_{\mathbb{M}}  = \mathrm{diam}(\bbM)$,
\item The equicontinuity assumption \cite[Assumption 3.5]{HNSW} is a consequence of 
the H{\" o}lder property (\ref{equicontinuity}).
\end{enumerate}

The case $\Omega \neq \mathbb M$ is more difficult, and the proof too long to be given here. It may be carried out by following the proofs of \cite[Lemma~B.1]{HNRW2} and \cite[Lemma~B.6]{HNRW2}, with appropriate modifications.

To establish \eqref{local_low_comp_eqn}, we begin by using \eqref{Sob_loc_full_error}, with $K\nu/2 +2d -4m-2:=\varepsilon>0$ and $\sigma=0$, to obtain $\|\chi_{\xi}-b_\xi\|_{L_p(\Omega)} \le \|\chi_{\xi}-b_\xi\|_{L_p(\mathbb M)} \le C'\rho^{4m+d}h^
{1+\varepsilon}\!$. \,
From this,  $\sum_{\xi\in\Xi}|a_{\xi}|\le (\#X)^{1-1/p}\|a\|_{\ell_p}\le C'q^{-d(1-1/p)}$, 
and
the triangle inequality, we have that, for $q_0$ sufficiently small, 
\[
q^{-d/p}\|\sum_{\xi\in\Xi}a_{\xi}(\chi_{\xi}-b_\xi)\|_{L_p(\Omega)}
\le C' \rho^{4m+d-\varepsilon-1}q_0^{\varepsilon} \|a\|_{\ell_p} 
\]
Again applying the triangle inequality and employing \eqref{low_comp_eqn}, we arrive at
\begin{gather*}
c(1-C'\rho^{4m+d-\varepsilon-1}q_0^{\varepsilon} ) \left\| \bsa  \right \|_{\ell_p(\Xi)}       
\le  
q^{-d/p} \|\sum_{\xi\in\Xi}a_{\xi}b_{\xi}\|_{L_p(\Omega)} \\ 
\le C(1+C'\rho^{4m+d-\varepsilon-1}q_0^{\varepsilon} )\left\| \bsa  \right \|_{\ell_p(\Xi)}.
\end{gather*}
Next, taking $q_0<1$, and (by increasing $K$ if necessary) 
$q_0^\varepsilon \le \frac{\rho^{-4m-d+\varepsilon+1}}{2C'}$, and using these in the the previous inequality results in \eqref{local_low_comp_eqn}. 
\end{proof}

\subsection{Inverse inequalities 
on $\Omega$}\label{main_results}
At this point we can prove the inverse inequality for both full and local Lagrange functions. 
We start with the full Lagrange functions.

\begin{theorem}\label{full_inverse}
Let $\Omega\subseteq\bbM$ be a compact domain satisfying Assumption \ref{Man}. 
Then for the Lagrange functions corresponding to $\kappa_m$,
there exist constants $C>0$ and $h_0>0$, so that for $h<h_0$
if $\Xi\subset \Omega$ has fill distance $h$,  mesh ratio $\rho$,
and ${\widetilde \Xi}\subset \bbM$ is a suitable extension of $\Xi$ 
(for instance, the one given by Lemma \ref{Extension}) 
then $V_{\Xi}\subset W_2^m(\Omega)$ and
for all $s = \sum_{\xi\in\Xi}a_{\xi}\chi_{\xi}\in V_{\Xi}$ and for $0\le \sigma\le m$,
we have
$$
\left\|s \right\|_{W_2^{\sigma}(\Omega)} 
\le   
C \rho^{m+d/2}  h^{-\sigma} \|s\|_{L_2(\Omega)}.
$$
\end{theorem}
\begin{proof}
From \eqref{Synthesis},  we have 
$\left\|s \right\|_{W_2^{\sigma}(\Omega)}
\le 
C \rho^{m} h^{d/2 -\sigma } \big\|\bsa \big\|_{\ell_2(X)}$, 
and from \eqref{low_comp_eqn}, with $p=2$ and $q=h/\rho$, we have
$c \left\| \bsa  \right \|_{\ell_2(\Xi)}       
\le  
h^{-d/2}\rho^{d/2} \|\sum_{\xi\in\Xi}a_{\xi}\chi_{\xi}\|_{L_p(\Omega)}$. 
Combining the two inequalities completes the proof.
\end{proof}

The proof for the local version is the same, except that we use \eqref{first_inverse_estimate} 
and \eqref{local_low_comp_eqn}.

\begin{theorem}\label{lowercomparison_omega}
Let $\Omega\subset\bbM$ be a compact domain satisfying Assumption \ref{Man}. 
Then for the local Lagrange functions corresponding to $\kappa_m$, 
with $K$ sufficiently large, we have that here exists a constant $h_0>0$, 
so that for $h<h_0$
if $\Xi\subset \Omega$ has fill distance $h$,  mesh ratio $\rho$,
and ${\widetilde \Xi}\subset \bbM$ is a suitable extension of 
$\Xi$ (for instance, the one given by Lemma \ref{Extension}) 
then 
for all $s = \sum_{\xi\in\Xi}a_{\xi}b_{\xi}\in {\widetilde V}_{\Xi}$ 
the following holds for all $0\le \sigma \le m$,  
$$
\left\|s\right\|_{W_2^{\sigma}(\Omega)} 
\le   
C \rho^{m+d/2}  h^{-\sigma} \|s\|_{L_2(\Omega)}.
$$
\end{theorem}
%

\section{Implications for Quasi-Interpolation and Approximation}\label{quasi_approx}
At this point, we are able to state several results that satisfactorily answer questions 
concerning interpolation, 
quasi-interpolation, and approximation properties of the spaces $V_\Xi$ and $\tilde{V}_\Xi$.
Some of these results have appeared previously in more restrictive settings while other results, 
such as pointwise error estimates for quasi-interpolation of continuous functions, are entirely new.

The first result is that the Lebesgue constant for interpolation is uniformly bounded.  
For the setting considered here (compact Riemannian manifolds and Sobolev-Mat{\'e}rn kernels),
this has been proven in \cite{HNW}. 
\begin{proposition}{\em (Lebesgue Constant, \cite[Theorem 4.6]{HNW})} \label{lebesgue_const}
 Suppose that $m>\tfrac{d}{2}$.  
 For a sufficiently dense set $\Xi\subset\mathbb{M}$ with mesh ratio $\rho$, 
 the Lebesgue constant 
$\Lambda := \sup_{\alpha\in\mathbb{M}}\sum_{\xi\in\Xi}|\chi_\xi (\alpha)|$, 
associated with the Sobolev-Mat\'ern kernel $\kappa_m$, 
is bounded by a constant depending only on $m$, $\rho$, and $\mathbb{M}$.
\end{proposition}

We remark that the key to proving this result is the pointwise exponential decay 
of the Lagrange function $\chi_\xi$, as given in \eqref{ptwise}. 
The same kind of bound also holds for local Lagrange functions. 
This can be shown by using the ``perturbation'' technique employed 
to prove \eqref{local_low_comp_eqn}. 
Similar results hold for other kernels on specific compact manifolds \cite{HNW-p}. 
In the case where the manifold is not compact, 
one typically is more interested in Lagrange functions based on finite point sets which are 
quasi-uniform with respect to a compact subset $\Omega\subset\mathbb{M}$. 
Nevertheless, a similar pointwise decay  estimate for Lagrange functions 
holds for that setting as well    \cite[Inequality 3.5]{HNRW2}.

There are two kinds of stability associated with the spaces $V_\Xi$ and $\widetilde V_\Xi$. 
The first concerns basis stability. In Proposition~\ref{lowercomparison}, 
we showed that both local and full Lagrange bases were very stable. 

The second kind of stability, which was established in \cite{HNSW}, 
concerns the $L_p$ norm of the $L_2$ projector. 
Let $W:\mathbb C^{\#\Xi} \to V(\kappa_m ,\Xi):=V_\Xi$ be a ``synthesis operator'' so 
$W:(a_\xi )_{\xi\in\Xi}\to\sum_{\xi\in\Xi}a_\xi v_\xi$ for a basis $(v_\xi )_{\xi\in\Xi}$ of $V_\Xi$.  
Likewise, let $W^* :L_1 (\mathbb{M})\to \mathbb C^{\#\Xi}$ 
be its formal adjoint $W^* :f\to(\langle f,v_\xi \rangle)|_{\xi\in\Xi}$.  
The $L_2$ projector is then 
\begin{equation}\label{L2-projector}
T_\Xi :=W(W^*W)^{-1}W^* :L_1 (\mathbb{M})\to V_\Xi
\end{equation}
in the sense that when $f\in L_2 (\mathbb{M})$, 
$T_\Xi f$ is the best $L_2$ approximant to $f$ from $V_\Xi$.

The $L_2$ norm of this projector is one -- because it is orthogonal -- 
while the $L_p$ and $L_{p^\prime}$ norms are equal because it is self-adjoint.  
Thus to estimate its $L_p$ operator norm $(1\leq p\leq\infty)$ it suffices to estimate its 
$L_\infty$ norm.

\begin{proposition}{\em ($L_2$ projector, \cite[Theorem 5.1]{HNSW})}\label{bdd_projector}
 For the Sobolev-Mat\'ern kernels, 
 for all $1\leq p\leq\infty$, the $L_p$ norm of the $L_2$ projector 
 $T_\Xi$ is bounded by a constant depending only on $\mathbb{M},\rho$ and $\kappa_m$.
\end{proposition}
For applications, the local Lagrange functions $\{b_\xi \}_{\xi\in\Xi}$ are substantially 
more computationally efficient than the full Lagrange functions.  
Nevertheless the bases $\{b_\xi \}_{\xi\in\Xi}$ and the space 
$\tilde{V}_\Xi =\mathrm{span}_{\xi\in\Xi}b_\xi$, under appropriate assumptions, 
enjoy essentially all the key properties as $\{\chi_\xi \}_{\xi\in\Xi}$ and $V_\Xi$ do.

In particular, Proposition~\ref{SobMat} shows that the spaces $V_\Xi$ and $\tilde{V}_\Xi$ 
can be quite close 
in Hausdorff distance and that the bases $\{b_\xi \}_{\xi\in\Xi}$ are slight perturbations of the bases 
$\{\chi_\xi \}_{\xi\in\Xi}$ even on compact subsets of the manifold.  
For the compact Riemannian manifold setting, under appropriate assumptions, 
the set $\{b_\xi \}_{\xi\in\Xi}$ is $L_p$ stable and each $b_\xi$ has pointwise 
polynomial decay of high order.  
This can be shown in the same way as in \cite[Thm 6.5]{FHNWW}.

A method to implement approximation from the space  $\tilde{V}_\Xi$
is by means of the quasi-interpolation operator
$$Q_\Xi f:=\sum_{\xi\in\Xi}f(\xi)b_\xi.$$
The quasi-interpolation operator provides $L_\infty$ convergence estimates at the same asymptotic rate 
as the interpolation operator.  Indeed
\begin{align*}
 |&I_\Xi f(x)-Q_\Xi f(x)|\leq\sum_{\xi\in\Xi}|b_\xi (x)-\chi_\xi (x)\| f(\xi)|\\
 &\leq C 
  (\#\Xi)\| f\|_{L\infty(\mathbb M)} \max_{\xi\in\Xi}\| b_\xi -\chi_\xi \|_{L\infty(\mathbb M)}.
\end{align*}
where  Lemma~\ref{compact_error} guarantees that $\| b_\xi -\chi_\xi \|_{L\infty(\mathbb M)}$ 
is as small as one likes 
depending on the ``footprint'' of $b_\xi$.  
Moreover the operators provide optimal $L_\infty$ approximation orders when the 
Lebesgue constant is uniformly bounded (see Proposition~\ref{lebesgue_const}).  
So, for example, it is shown in \cite[Cor 5.9]{HNW-p} that restricted surface spline interpolation satisfies 
$\| I_\Xi f-f\|_{L\infty(\mathbb M)} \leq Ch^\sigma$ for $f\in C^{2m}(\mathbb{S}^2 )$ when 
$\sigma=2m$ and $f\in B^{\sigma}_{\infty,\infty}(\mathbb{S}^2 )$ for $\sigma\leq 2m$.  
Thus $Q_\Xi$ inherits the same rate of approximation.

The quasi-interpolation operator also provides two more useful approximation properties.  
The first deals with pointwise error estimates for continuous functions.  
In the early 1990's, Brown \cite{A_L_Brown} showed that, for several classes of RBFs, if the density
parameter $h_\Xi$ decreased to zero for point sets $\Xi$ in compact $\Omega$, then
\[
\text{dist}_\infty (f,V_\Xi )\to 0
\]
for any continuous function $f$.  The argument given was nonconstructive.  
The next result gives pointwise error estimates when approximating an arbitrary continuous function 
$f$ on $\mathbb{M}$ in terms of its modulus of continuity.  
The result is reminiscent of a similar one for univariate splines.
%

\subsection{Approximation rates based on local smoothness}\label{local_rates}
For a function $f$,  the global modulus of continuity  is defined as 
$\omega(f,t):=\max_{|x-y|\leq t}|f(x)-f(y)|$, 
and  the modulus of continuity at $x_0$ is $\omega(f,t,x_0 ):=\max_{|x-x_0 |\leq t}|f(x)-f(x_0 )|$. 
Recall also the constants $K$ and $J$, discussed in Section~\ref{local_lagrange}, 
and that  $\Lambda$ denotes the Lebesgue constant.


%
\begin{theorem}\label{local}
 Assume the conditions and notation of Theorem \ref{main_local_bernstein} and 
 Proposition~\ref{lebesgue_const} hold.  
Then for each $x_0 \in\mathbb{M},\,f\in C(\mathbb{M})$ with $\| f\|_{L_\infty(\mathbb M)}  =1$, 
the following hold.
 \begin{enumerate}
  \item[i)] $|f(x_0 )-I_\Xi f(x_0 )|\leq\max\{\Lambda\omega(f,Kh |\log h|),2h^{J - 1} \}$
  \item[ii)] $\| f-I_\Xi f\|_{L_\infty(\mathbb M)} \leq\Lambda(K + 1)\omega(f,h|\log h|)$
  \item[iii)] $|f(x_0 )-Q_\Xi f(x_0 )|\leq\max\{\Lambda\omega(f,Kh|\log h|,x_0 ),h^{J-2}\}$
  \item[iv)] $\| f-Q_\Xi f\|_{L_\infty(\mathbb M)}=O(\| f-I_\Xi f\|_{L_\infty(\mathbb M)} )$
 \end{enumerate}
\end{theorem}
\begin{proof}
 Note that
 $$|f(x_0 )-I_\Xi f(x_0 )|\leq\sum_{\xi\in\Xi}\bigl|f(x_0 )-f(\xi )\bigr||\chi_\xi (x_0 )|+
C \| f\|_{L\infty(\mathbb M)} |1-\sum_{\xi\in\Xi}\chi_\xi (x_0)|.$$
 By Corollary \ref{lebesgue_lemma} below, 
 $|1-\sum_{\xi\in\Xi}\chi_\xi (x)|=\mathcal O(h^{2m})$. 
 Define the ball $B_{x_0 }$ by $B_{x_0 }:=B(x_0 ,Kh\ln h^{-1})$ and let its complement be denoted 
  $B^{C}_{x_0}:=\mathbb{M}\backslash B_{x_0}$.  
 Then
 \begin{align*}
  |f(x_0 )-I_\Xi f(x_0 )|
  &\leq\sum_{\xi\in B_{x_0}\cap\Xi}|f(\xi )-f(x_0 )|\,|\chi_\xi (x_0 )|\\
  &+\sum_{\xi\in B^{C}_{x_0}\cap\Xi}|f(x_0 )-f(\xi )|\,|\chi_\xi (x_0 )|+C\|f\|_{L_\infty(\mathbb M)}h^{2m}\\
  &\leq\max_{\xi\in B_{x_0}\cap\Xi}|f(\xi )-f(x_0 )|\sum_{\xi\in B_{x_0}\cap\Xi}|\chi_\xi (x_0 )|\\
  &\qquad\quad+
  C\| f\|_{L\infty(\mathbb M)} \big(\sum_{\xi\in B^{C}_{x_0}\cap\Xi}\big(1+\tfrac{\text{\rm dist}(x_0 ,\xi)}{h}\big)^{-J}+h^{2m}\big)\\
  &\leq\Lambda\omega(f,Kh|\log h|)(x_0 )+C\max(h^{J-1},h^{2m})\|f\|_{L_\infty(\mathbb M)} .
 \end{align*}
 The second inequality follows from
 $\omega(f,Kt)\leq(K + 1)\omega(f,t)$
 and the fact that if $\omega(f,t)/t\to 0$ as $t\to 0$, then $f$ is a constant  \cite{DL}.  Inequality iii) follows from
 \begin{align*}
  |f(x_0 )-Q_\Xi f(x_0 )|
  &\leq|f(x_0)-I_\Xi f(x_0 )|+|I_\Xi f(x_0 )-Q_\Xi f(x_0 )|\\
  &\leq|f(x_0)-I_\Xi f(x_0 )|+\sum_{\xi \in \Xi}|f(\xi)|\,|\chi_\xi (x_0 )-b_\xi (x_0 )|\\
  &\leq|f(x_0)-I_\Xi f(x_0 )|+\| f\|_{L\infty(\mathbb M)} \sum_{\xi \in \Xi} \|\chi_\xi -b_\xi \|_\infty\\
  &\leq\Lambda\omega(f,Kh|\log h|)(x_0 )+Ch^{J}(\#\Xi)\\
  &\leq\Lambda\omega(f,Kh|\log h|)(x_0 )+C\rho^d h^{J-d}.
 \end{align*}
 The last inequality is clear.
\end{proof}

We remark that the pointwise estimate in the first inequality above requires only continuity at a single point,
and boundedness elsewhere. 

\subsection{Rates for functions with higher smoothness}\label{SS:abstract_approx}
By the global boundedness of the Lebesgue constant, we know that interpolation is ``near-best''.
Similarly, by Theorem~\ref{local}\,(iv), quasi-interpolation is near-best as well. In this subsection, we 
establish precise rates of decay $\text{dist}_{\infty}(f,S(\Xi))$. This is established using an approximation
scheme similar to the one employed in \cite{DR} -- 
it uses the fact that the kernel is a fundamental solution  
for $\mathcal{L} = \sum_{j=0}^m (\nabla^j)^*\nabla^j$ (pointed out in Section \ref{SSS:Kernels})
to obtain the identity
$
f(x) 
= 
\int_{\mathbb{M}} \mathcal{L} f(\alpha) \kappa_m(x,\alpha) \mathrm{d} \alpha
$ 
for $f\in C^{2m}(\mathbb{M})$. 
As in \cite{DR}, for every $\alpha \in \mathbb{M}$, 
we use a modified kernel $\tilde{\kappa}(x,\alpha)$ constructed from $\Xi$ by taking 
$\tilde{\kappa}(\cdot,\alpha)\in S(\Xi)$, with coefficients depending continuously on $\alpha$. 
We may then replace $\kappa_m$ by $\tilde{\kappa}$ in the reproduction formula for $f$.

For $\alpha\in\mathbb{M}$, define $\Xi_{\alpha}$ as follows:
$$
\Xi_{\alpha} 
:= 
\begin{cases}
\Xi\cup \{\alpha\},& \text{dist}(\alpha,\Xi) \ge  h/2\\
 \Xi \cup \{\alpha\}\setminus \{\xi^*\}, &  \text{dist}(\alpha,\Xi) \le  h/2
 \end{cases}
$$
where $\xi^*$ is the nearest point of $\Xi$ to $\alpha$.
For this point set, we have the fill distance $h(\Xi_{\alpha},\mathbb{M}) \le 3h/2$ and 
separation radius $q(\Xi_{\alpha})\ge \min(q,h/2)$.

For every $\alpha\in \bbM$, we consider the Lagrange function $\lambda_{\alpha}\in S(\Xi_{\alpha})$ 
centered at $\alpha$.
We express this Lagrange function as
$\lambda_{\alpha} = \sum_{\xi\in\Xi_{\alpha}} A_{\alpha,\xi}\kappa_m(\cdot, \xi)$.
Let 
$
a(\xi,\alpha) 
:=
- A_{\alpha,\xi}/A_{\alpha,\alpha}
$ 
for $\xi \in \Xi_{\alpha}\setminus \{\alpha\}$.
The approximation scheme is given by way of the operator
 $$S_{\Xi} f 
 := \sum_{\xi\in\Xi}c_{\xi}\kappa_m(\cdot,\xi)  
 $$
 with 
 $
 c_{\xi} 
 = 
 \int_{\mathbb{M}} \mathcal{L} f(\alpha) a(\xi,\alpha) \mathrm{d} \alpha
 $.
 
This works because the kernel $\kappa_m$ used  in the reproduction of smooth $f$ 
can be replaced by a modified kernel
$
\tilde{\kappa}(x,\alpha) 
= 
\sum_{\xi\in\Xi_{\alpha}, \xi \ne \alpha} a(\xi,\alpha) \kappa_m(\cdot,\xi)
$,
which is a linear combination of the original kernel sampled from $\Xi_{\alpha}$. 
We measure the difference of the two kernels as:
\begin{eqnarray*}
\mathrm{err}(x,\alpha) 
&:= & 
\kappa_m(x,\alpha) - \tilde{\kappa}(x,\alpha)\\
&=&
\kappa_m(x,\alpha) - \sum_{\substack{\xi \in \Xi_{\alpha}\\ \xi \ne \alpha}} a(\xi,\alpha)  \kappa_m(\cdot,\xi) \\
&= &
\frac{1}{A_{\alpha,\alpha}}\lambda_{\alpha}(x).
\end{eqnarray*}

To further control this error, we estimate $|A_{\alpha,\alpha}|$ from below. 
We do this by applying the zeros lemma for balls \cite{HNW-p} on
the set $B(\alpha, Mh)$ 
(for a sufficiently large constant $M$ - a constant which depends only on $\mathbb{M}$ and $m$). 
Thus, we have
\begin{eqnarray}\label{replacement_first}
|\lambda_{\alpha}(\alpha)| 
&\le& 
\|\lambda_{\alpha}\|_{L_{\infty}(B(\alpha,Mh))} \nonumber\\
&\le& 
C \bigl(Mh\bigr)^{m-d/2}\|\chi_{\alpha}\|_{W_2^m(\mathbb{M})} \nonumber\\
&=& 
C h^{m-d/2}
|\langle \chi_{\alpha},\chi_{\alpha}\rangle |^{1/2}
\end{eqnarray}
 Replacing $\chi_{\alpha}(\alpha)$ with $1$ and 
 $\langle \chi_{\alpha},\chi_{\alpha}\rangle$ 
 with 
 $|A_{\alpha,\alpha}|$,
 we have a lower bound for $|A_{\alpha,\alpha}|$. 
 Namely, there is a constant $C>0$ depending only on $m,\mathbb{M}$ so that 
 \begin{equation}\label{replacement_second}
 |A_{\alpha,\alpha}| 
 \ge C h^{d-2m}
  \end{equation}

Combining (\ref{replacement_first}), (\ref{replacement_second}) and the pointwise decay rates
for the Lagrange functions, we obtain the bound
 \begin{equation}\label{error_kernel1}
 |\mathrm{err}(x,\alpha)| 
 \le 
 C  
 \rho^{m-d/2}
h^{2m-d} 
 e^{-\nu  \left(\frac{\mathrm{dist}(x,\alpha)}{ h}\right)}.
 \end{equation}

 At this point, we have the following result for approximation of smooth functions.

 \begin{theorem}\label{topsmoothness}
 For $1\le p<\infty$ and 
 $f\in W_p^{2m}(\mathbb{M})$,
 or $f\in C^{2m}(\mathbb{M})$ when $p=\infty$,
 there is a constant $C<\infty$ 
 depending only on $m$ and $\mathbb{M}$ so that
 $$
 \left\|{f - S_{\Xi}f}\right\|_{L_p(\mathbb M)} 
 \le 
 C h^{2m}\|f\|_{W_p^{2m}(\mathbb{M})}
 $$
  \end{theorem}
\begin{proof}
The 
error $\left\|{f - S_{\Xi}f}\right\|_{p}$ is
bounded by the norm of the integral operator 
$ \mathrm{Err}: g \mapsto \int_{\mathbb{M}} g(\alpha) |\mathrm{err}(\cdot,\alpha) | \mathrm{d} \alpha$, 
which has non-negative
kernel $|\mathrm{err}(x,\alpha)|$.
 Indeed, we have 
 $
 |f(x) - S_{\Xi}f(x)| 
 \le 
 \int_{\mathbb{M}}| \mathcal{L} f(\alpha) | \, |\mathrm{err}(x,\alpha)| \, \mathrm{d} \alpha
 $, 
 so
 $$
 \|f - S_{\Xi}f\|_{L_p(\mathbb M)} 
 \le 
 \|\mathcal{L} f\|_{L_p(\mathbb M)} \| \mathrm{Err}\|_{L_p\to L_p}.
 $$
 We estimate the norm of this operator on $L_1$ and $L_{\infty}$ -- the $L_p$ result then follows  by interpolation.
In other words, 
$ 
\|\mathrm{Err}\|_{1\to1} 
\le 
\max_{\alpha\in\mathbb{M}} \int_{\mathbb{M}} |{\mathrm{err}(x,\alpha)}| \mathrm{d} x
$ 
and 
$
\|\mathrm{Err}\|_{\infty \to\infty} 
\le 
\max_{x\in\mathbb{M}} \int_{\mathbb{M}} |{\mathrm{err}(x,\alpha)}| \mathrm{d} \alpha
$.
Using (\ref{error_kernel1}) and symmetry, both are bounded by
\begin{eqnarray}
 C\rho^{m-d/2} h^{2m-d}\max_{\alpha\in \mathbb{M}}
 \int_{\mathbb{M}} e^{-\nu  \left(\frac{\mathrm{dist}(x,\alpha)}{ h}\right)} \mathrm{d} x
\le C \rho^{m-d/2} h^{2m}
\end{eqnarray}
and the theorem follows.
 \end{proof}

 A result for lower smoothness is also possible.
 Let us define the Besov space  $B_{p,\infty}^{\sigma}(\mathbb{M})$ as a real interpolation space 
 between $L_p(\mathbb{M})$ and $W_p^{2m}(\mathbb{M})$.
Let $B_{p,\infty}^{\sigma}(\mathbb{M})$  be the set of (equivalence classes) of functions  
 $f\in L_p(\mathbb{M})$ for which the expression
 \begin{equation}\label{besov_def}
 \|f\|_{B_{p,\infty}^{\sigma}(\mathbb{M})}
 := 
 \sup_{t>0} t^{-\sigma/2m}
 \inf_{g\in W_{p}^{2m}(\mathbb{M})} 
 \bigl(\|f-g\|_{L_p(\mathbb{M})}+ t\|g\|_{W_{p}^{2m}(\mathbb{M})}\bigr)
 \end{equation}
 is finite. 
 (When $p=\infty$, we make the the usual replacements  of $ L_p(\mathbb{M})$ by  
 $C(\mathbb{M})$ and
 $W_p^{2m}(\mathbb{M})$ by $C^{2m}(\mathbb{M})$.)
 That this is a Banach space and the above is a norm can be found in \cite{adams-f}, \cite{DL} 
 or \cite{triebel1992}.
 We note in particular that \cite{triebel1992} shows this definition is equivalent to  other standard, intrinsic 
 constructions of Besov spaces on manifolds, 
 and relates these to the Sobolev scale and other families of smoothness spaces. 
 Of special interest is the case of the H{\"o}lder spaces with fractional exponent:
 $C^{\sigma}(\mathbb{M}) = B_{\infty,\infty}^{\sigma}(\mathbb{M})$.
 \begin{theorem}
 Let $f\in B_{p,\infty}^{\sigma}(\mathbb{M})$ for $1\le p\le \infty$ and $0<\sigma \le 2m$. Then we have
 $$
 \mathrm{dist}_{p,\mathbb{M}}(f, S(\Xi))
 \le 
 Ch^{\sigma} \|f\|_{B_{p,\infty}^{\sigma}(\mathbb M)}.
 $$
 \end{theorem}
 \begin{proof}
 The follows from a standard $K$-functional argument, 
 by splitting $f = g+(f-g)$, with $g\in W_p^{2m}(\mathbb{M})$ (or $C^{2m}(\mathbb{M})$) 
 and $f-g\in L_p(\mathbb{M})$ (or $C(\mathbb{M})$). 
 In particular, for $h>0$, set $t=h^{2m} $.
 and find $g$ so that
 $$
 \|f-g\|_{L_p(\mathbb{M})} + t \|g\|_{W_p^{2m}(\mathbb{M})}
 \le 2 t^{\sigma/2m}  \|f\|_{B_{p,\infty}^{\sigma}(\mathbb{M})}. 
 $$
 This ensures that 
 $$
  \|f-g\|_{L_p(\mathbb{M})} 
  \le 
  2 h^{\sigma} \|f\|_{B_{p,\infty}^{\sigma}(\mathbb{M})} 
  \quad 
  \text{ and }
  \quad
  \|g\|_{W_p^{2m}(\mathbb{M})}
  \le 
  2 h^{\sigma-2m} \|f\|_{B_{p,\infty}^{\sigma}(\mathbb{M})}.
 $$
 Finally, we take $S_{\Xi} g$ as our approximant to $f$, obtaining the desired result by applying
 the triangle inequality and Theorem~\ref{topsmoothness}.
 \end{proof}

A drawback of the previous results in this section  is that 
the approximation scheme $S_{\Xi}$ is not easy to implement.  
The good news is that the stability of the schemes $I_\Xi$, $Q_\Xi$  and $T_{\Xi}$ 
imply that these operators
 inherit the same convergence rate.
 This is a consequence of the Lebesgue constants 
being bounded (Proposition~\ref{lebesgue_const}) and the small error between $I_\Xi$ and $Q_\Xi$.

\begin{corollary}\label{lebesgue_lemma}
There exists a constant $C>0$ so that
for $0<\sigma\le 2m$ and $f\in C^{\sigma}(\mathbb{M})$, 
we have
$$
\|f- I_{\Xi}f \|_{L_\infty(\mathbb M)} \le C h^{\sigma} \|f\|_{C^{\sigma}(\mathbb{M})} \ \text{\rm and }
\|f- Q_{\Xi}f \|_{L_\infty(\mathbb M)} \le C h^{\sigma} \|f\|_{C^{\sigma}(\mathbb{M})}.
$$
For $0<\sigma\le 2m$, $1\le p\le \infty$ and 
$f\in B_{p,\infty}^{\sigma}(\mathbb{M})$ or $f\in W_p^{2m}(\mathbb{M})$ when $\sigma=2m$, 
we have
$$\|f- T_{\Xi}f \|_{L_p(\mathbb M)} \le C h^{\sigma} \|f\|_{B_{p,\infty}^{\sigma}(\mathbb{M})},$$
where $T_\Xi$ is the least-squares projector defined in \eqref{L2-projector}.
\end{corollary}

\subsection{Approximation on bounded regions}
As a final note, we observe that the approximation power of spaces $S(X)$ on $\bbM$, 
where $X$ is dense in $\bbM$,
extends to the setting of approximation over a compact domain $\Omega\subset \bbM$ having a Lipschitz boundary and satisfying Assumption~\ref{Man}, using $\tilde{V}_{\Xi}$, with $\Xi$ 
dense in the union of $\Omega$ and an ``annulus'' around $\Omega$.

Our final result shows that optimal $L_\infty$ approximation rates, when approximating a smooth function $f$ on $\Omega$, can be obtained from data sites 
either inside or ``close'' to $\Omega$.  The result illustrates the local nature of the basis $\{ b_\xi \}$. 

Let $f\in C^\sigma (\Omega)$, where $\sigma>0$ is an integer, and let $\tilde f \in C^\sigma (\mathbb{M})$ be a smooth extension of $f$ to $\mathbb{M}$, i.e., $\tilde f |_\Omega =f|_\Omega$.  Suppose that $\mathcal{A}=\{x\in\mathbb{M}\backslash\Omega:\text{ dist}(x,\Omega)\leq Kh|\log h|\}$ and that $\Xi$ is a finite set contained in $\Omega\cup\mathcal{A}$, with fill distance $h$.  In addition, let $\widetilde \Xi$ be a quasi-uniform extension of $\Xi$ to all of $\bbM$, as given in Lemma \ref{Extension}.  Finally, let $\kappa_m$ be a kernel as described in Section~\ref{SSS:Kernels} with associated spaces
\[
\tilde{V}_{\Xi}=\mathrm{span}_{\xi\in \Xi}b_\xi\ \text{ and }\ \tilde{V}_{\widetilde \Xi}=\mathrm{span}_{\xi\in \widetilde \Xi}b_\xi.
\]

\begin{theorem} If $\sigma\le 2m$, then 
$$
\text{\rm dist}_{\infty,\Omega}(f,\tilde{V}_X )\sim \text{\rm dist}_{\infty,\mathbb{M}}(f_e ,\tilde{V}_{\widetilde X}) \le \left\{
\begin{aligned}
 Ch^\sigma\|f\|_{C^\sigma(\Omega)}\\
Ch^\sigma \|f_e\|_{C^\sigma(\mathbb M)}.
\end{aligned}\right.
$$	
\end{theorem}
\begin{proof} By the global boundedness of the Lebesgue constant $\Lambda$, we know that interpolation is near-best approximation. Similarly, by Theorem~\ref{local}(iv), quasi-interpolation is near-best approximation as well. Hence, with $J=K\nu/2 - 2m +d$, we have that
\begin{align*}
\max_{x \in \Omega}|f(x) - &\sum_{\xi \in \Xi}\tilde f(\xi)b_{\xi}(x)|
\leq 
\max_{x \in \Omega}|f(x) - \sum_{\xi \in \Xi_e}f_e(\xi)b_{\xi}(x)| \\
& \qquad+\max_{x \in \Omega}\sum_{\xi\in \widetilde \Xi\setminus\Xi}|\tilde f(\xi)b_{\xi}(x)| \\
&\leq  
\max_{x \in \Omega}|f(x) - \sum_{\xi \in \widetilde \Xi}\tilde f(\xi)b_{\xi}(x)| +
\| f \|_{\infty}\sum_{\xi\in \widetilde \Xi\setminus\Xi}\left(1+\tfrac{\text{\rm dist}(x_0 ,\xi)}{h}\right)^{-J}\\
&\leq  \max_{x \in \Omega}|f(x) - \sum_{\xi \in \widetilde \Xi}\tilde f(\xi)b_{\xi}(x)| + 
\| f\|_{\infty}h^{J - 1}
\end{align*}
where $h^J$ can be chosen small compared to the first term, because, by \eqref{general_norm}, the parameter $K$ in $J$ can be chosen large enough for this to happen. The theorem then follows from Corollary~\ref{lebesgue_lemma}.
\end{proof}

We remark that one only needs to have \emph{local} information in a small annulus outside $\Omega$ to obtain full approximation order.  Moreover, as previously discussed, approximation order on manifolds is often known.

\section{Volume Comparisons }\label{appendix_vol_comp}

\begin{proposition}\label{manifold}
We assume that $\bbM$ is a closed, compact, connected $d$-dimensional $C^\infty$ Riemannian manifold. 
There exist constants $0<\alpha_\bbM<\beta_\bbM<\infty$  
so that any ball $B(x,r)$ 
satisfies
\begin{equation}\label{balls2}
\alpha_\bbM r^d \le \mathrm{vol}(B(p,r))\le \beta_\bbM r^d
\end{equation}
for all $0\le r \le \rmd_\bbM$.
\end{proposition}

\begin{proof}
By Property~\ref{bounded_geometry}, $\bbM$ has bounded geometry, so the Ricci curvature, 
$\mathrm{Ric}$, is bounded below. 
Hence, there is a $k\in \RR$ such that $\mathrm{Ric}\ge (d-1)k$. 
Let $\bbM^d_k$ denote one of the canonical manifolds (sphere, $\RR^d$, hyperbolic space) having constant sectional curvature $k$. 
In addition, let $p\in \bbM$, $\tilde p\in \bbM^d_k$, $V_r:= \mathrm{vol}(B(p,r))$ and $V_r^k :=\mathrm{vol}(B^k(\tilde p,r))$. 
The Bishop-Gromov Comparison theorem states that the ratio $V_r/V_r^k$ is non increasing and, 
as $r\downarrow  0$, $V_r/V_r^k \to 1$, no matter which $p,\tilde p$ are chosen.  
Since $\mathrm{Ric}$ may become negative, 
we can handle all of the cases at once by assuming that $k<0$, 
which means that $\bbM^d_k$ is a hyperbolic space. 

The model that we take for $\bbM^d_k$ will be the Poincar\'e ball, 
so that $\bbM^d_k = \{ x=(x^1, \ldots ,x^d) \in \RR^d \mid \|x\|_2^2 < -4/k\}$. 
Let $A:=1+(k/4)\|x\|_2^2$. 
In these coordinates, the Riemannian metric is given by $g_{jk} = \delta_{jk}/A^2$; equivalently, 
$ds^2 = \sum_{j=1}^d (dx^j)^2/A^2$. 
We want to introduce geodesic normal coordinates, centered at $x^j=0$, $j=1,\ldots,d$. 
Let $t\ge 0$ and set $x^j = \frac{2}{\sqrt{|k|}} \tanh(\sqrt{|k|}t/2)\xi^j$, 
where $\xi = (\xi^1,\ldots,\xi^d)\in \S^{d-1}$. 
A straightforward computation shows that 
\begin{equation}\label{arc_length_k}
ds^2 = dt^2 + \frac{1}{|k|}\sinh^2(\sqrt{|k|}t)ds^2_{\S^{d-1}},
\end{equation}
where $t$ the length of the geodesic joining the origin to $x^j= t\xi^j$ .
It follows that the volume element in these coordinates is
\begin{equation}\label{volume_element_k}
d\mu_k = \frac{1}{\sqrt{|k|}^{\,d-1}}\sinh^{d-1}(\sqrt{|k|}t)dt d\mu_{\S^{d-1}},
\end{equation}
and, consequently,
\begin{equation}
\label{volume_k}
V_r^k=\mathrm{vol} (B^k(\tilde p, r)) = \frac{1}{\sqrt{|k|}^{\,d-1}}\omega_{d-1} \int_0^r \sinh^{d-1}(\sqrt{|k|}t)dt.
\end{equation}

We will need bounds on $V_r^k$ for $r\le R$, where $R$ is fixed. 
These are easy to obtain, since $1\le \frac{\sinh(x)}{x}\le \frac{\sinh(X)}{X}$ for all $0\le x\le X$. 
Just take $X=\sqrt{|k|}R$:
\[
1\le \bigg(\frac{\sinh(\sqrt{|k|}t)}{\sqrt{|k|}t}\bigg)^{d-1}
\le 
\bigg(\frac{\sinh(\sqrt{|k|}R)}{\sqrt{|k|}R}\bigg)^{d-1} := \beta_{d,k,R}.
\]
Multiplying both sides by $\omega_{d-1}t^{d-1}$ and integrating results in this inequality:
\[
\frac{\omega_{d-1}}{d} r^d 
\le 
\int_0^r \omega_{d-1}\bigg(\frac{\sinh(\sqrt{|k|}t)}{\sqrt{|k|}t}\bigg)^{d-1}t^{d-1}dt 
\le 
\beta_{d,k,R} \frac{\omega_{d-1}}{d} r^d.
\]
Using this in \eqref{volume_k} results in 
\begin{equation}
\label{volume_k_bounds}
\frac{\omega_{d-1}}{d} r^d \le V_r^k \le \beta_{d,k,R} \frac{\omega_{d-1}}{d} r^d.
\end{equation}

We can now employ the Bishop-Gromov Theorem to obtain \eqref{balls}. 
Since $V_r/V_r^k$ is non increasing and tends to 1 as $r\downarrow 0$, 
we have that $V_r\ge V_r^k \ge \frac{\omega_{d-1}}{d} r^d $. 
Also, we have that $V_r/V_r^k\le V_{\rmd_\bbM}/V_{\rmd_\bbM}^k$. 
Thus, $V_r\le \big(V_{\rmd_\bbM}/V_{\rmd_\bbM}^k)V_r^k$. 
Employing \eqref{volume_k_bounds} in conjunction with these inequalities yields
\[
\frac{\omega_{d-1}}{d} r^d 
\le 
V_r \le \beta_{d,k,{\rmd_\bbM}} \frac{\omega_{d-1}}{d} \big(V_{\rmd_\bbM}/V_{\rmd_\bbM}^k\big)r^d.
\]
We want to refine this. 
To do that, we begin by observing that $\overline{B(p,\rmd_\bbM)} = \bbM$, 
because no point in $\overline{B(p,\rmd_\bbM)}$ is at a distance from $p$ greater than the diameter $\rmd_\bbM$; 
thus, $V_{\rmd_\bbM}=\mathrm{vol}(\bbM)$. 
Next, by \eqref{volume_k_bounds}, $V_{\rmd_\bbM}^k\ge \frac{\omega_{d-1}}{d} \rmd_\bbM^d$. 
Finally, using these in the inequality above yields
\[
\frac{\omega_{d-1}}{d} r^d \le V_r \le \beta_{d,k,\rmd_\bbM} \rmd_\bbM^{-d}\mathrm{vol}(\bbM)r^d,
\]
from which \eqref{balls} follows with $\alpha_\bbM=\frac{\omega_{d-1}}{d} $ 
and $\beta_\bbM =\beta_{d,k,\rmd_\bbM} \rmd_\bbM^{-d}\mathrm{vol}(\bbM)$. 
\end{proof}


%
%

\end{document}